\documentclass[11pt]{article}
\usepackage{geometry}
\usepackage{sectsty}

\usepackage{amsmath}
\usepackage{amssymb}
\usepackage{mathtools} 
\usepackage{amsthm}
\usepackage{thmtools, thm-restate}
\usepackage[hidelinks]{hyperref}
\usepackage[nameinlink]{cleveref}
\usepackage{enumerate}

\usepackage{tikz}
\usetikzlibrary{calc}
\usetikzlibrary{intersections}
\usetikzlibrary{arrows.meta}
\usetikzlibrary{patterns}
\usetikzlibrary{angles}
\usetikzlibrary{shapes.geometric}
\usetikzlibrary{math}

\usepackage[obeyDraft]{todonotes} 
\usepackage{soul} 

\usepackage{color}
\usepackage{graphicx}

\newcommand{\reals}{\ensuremath{\mathbb{R}}}
\newcommand{\norm}[1]{\ensuremath{\left\vert\left\vert#1\right\vert\right\vert}}

\newcommand{\flow}{\ensuremath{\mathcal{F}}}

\newcommand{\uT}[1]{\ensuremath{T_1 #1}}
\newcommand{\genflow}{\ensuremath{\mathfrak F}}

\newcommand{\pr}{\ensuremath{\textrm{pr}}}
\newcommand{\z}[2]{\ensuremath{\mathcal{Z}_{#1}^{(#2)}}}
\newcommand{\zn}{\ensuremath{\z{n}{N}}}
\newcommand{\zN}[1]{\ensuremath{\z{#1}{N}}}

\newcommand{\sT}{\ensuremath{\mathcal{T}}}
\newcommand{\TS}{\ensuremath{T^+\partial S}}
\newcommand{\TTS}{\ensuremath{\TS\backslash (\cup_i\Xi_i)}}
\newcommand{\TTTS}{\ensuremath{\TS\backslash (\cup_i\Xi_i\cup\Trap{\partial S}^{(K)})}}
\newcommand{\Trap}[1]{\ensuremath{\textrm{Trap}(#1)}}

\newcommand{\mindist}{\ensuremath{d_{\min}}}
\newcommand{\minextra}{\ensuremath{d^*}}
\newcommand{\maxdist}{\ensuremath{D}}
\newcommand{\maxsec}{\ensuremath{sec_{\max}}}
\newcommand{\mincurv}{\ensuremath{\kappa_{\min}}}
\newcommand{\minref}{\ensuremath{\xi}}
\newcommand{\globalcurv}{\ensuremath{\Theta}}
\newcommand{\frontmincurv}{\ensuremath{k_{\min}}}

\newcommand{\comment}[1]{}

\newtheorem{theorem}{Theorem}
\newtheorem{proposition}{Proposition}
\newtheorem{lemma}[proposition]{Lemma}
\newtheorem{corollary}{Corollary}[proposition]

\theoremstyle{definition}

\crefname{condition}{condition}{conditions}

\title{Rigidity of Travelling Times for Strictly Convex Obstacles in Riemannian Manifolds}

\author{Tal Gurfinkel\footnote{This research is supported by an Australian Government Research Training Program (RTP) Scholarship.}\ \footnote{tal.gurfinkel@research.uwa.edu.au}\ , Lyle Noakes\footnote{lyle.noakes@uwa.edu.au}\ , Luchezar Stoyanov\footnote{luchezar.stoyanov@uwa.edu.au}\\\\ Department of Mathematics and Statistics, \\University of Western Australia, Crawley, 6009, WA, Australia}

\begin{document}

\maketitle

\begin{abstract}
\noindent
	Let $K$ and $L$ be two disjoint unions of strictly convex obstacles contained within a Riemannian manifold with boundary $S$ of dimension $m\geq 2$. The sets of travelling times $\sT_K$ and $\sT_L$ of $K$ and $L$, respectively, are composed of triples $(x,y,t)\in\partial S\times\partial S\times\reals^+$ where $t$ is the length of a billiard trajectory with endpoints $x$ and $y$ that reflects elastically on $K$ (or $L$ for $(x,y,t)\in\sT_L$). In \cite{GNS2023Preprint1} it was shown that (under some natural curvature bounds on $S$) if $\sT_K=\sT_L$ and $K$ and $L$ were equivalent up to tangency then $K = L$. In this paper we remove this requirement for $K$ and $L$, and show that if $\sT_K = \sT_L$ then $K = L$ whenever $m\geq 3$.
\end{abstract}

\begin{quote}
\end{quote}

\begin{quote}
	{\bf MSC Classification (2020):} 37D40, 37C83, 53C21.
\end{quote}

\noindent
Classically, mathematical billiards are dynamical systems defined by the motion of a point particle travelling along piecewise paths within a submanifold $\Omega\subseteq\reals^m$ with boundary of Euclidean space. The particle's motion is a straight line within the interior of $\Omega$ and upon contact with $\partial\Omega$ it satisfies the classical law of reflection ``the angle of incidence is equal to the angle of reflection". In this paper we will deal with dispersing billiards (i.e. billiards with strictly convex boundary), which were originally studied in the plane by Sinai (\cite{MR0274721}). Dispersing billiards are one of many billiard systems which exhibit chaotic behaviour. Chaotic billiard systems have been studied extensively since the introduction of the Sinai billiard (cf. \cite{CHERNOVCHAOTICBILLIARDS}). One can define billiards on a Riemannian manifold with boundary in an analogous way, by requiring the motion of the system to follow along geodesics and satisfy the same law of elastic reflection (see e.g. \cite{MR807598}). We shall study an inverse problem in billiards on a Riemannian manifold with boundary $S$, containing strictly convex obstacles in its interior. We consider the boundaries of the obstacles as the boundary on which the billiard rays reflect, while the rays do not reflect on the boundary of the exterior manifold $S$. This ensures a finite horizon, i.e. a maximal length for any free ray. Other similar inverse problems include rigidity problems concerning the marked length spectrum, the broken X-ray transform and more \cite{MR3990606, MR3514564, MR4276284, GEOMETRICINVERSEPROBLEMS}.
The set $\sT_K$ of travelling times of an obstacle $K$ is composed of the lengths and endpoints of billiard trajectories reflecting elastically on $K$ which begin and end on the boundary of a manifold $S$ containing $K$. In \cite{GNS2023Preprint1} it was shown that for pairs $(K,L)$ of unions of disjoint strictly convex obstacles which are equivalent up to tangency\footnote{$K$ and $L$ are said to be \emph{equivalent up to tangency} if for any inward vector $\sigma$ of $\partial S$ the billiard trajectories generated by $\sigma$ will coincide for $K$ and $L$ up until a first point of tangency with $\partial K$ or $\partial L$ respectively.}, if $\sT_K=\sT_L$ then $K=L$. Note that the class of pairs of obstacles which are equivalent up to tangency is non-trivial, in fact it contains those pairs which satisfy Ikawa's non-eclipse condition \cite{MR949013} (or equivalently the general position condition) and have the same set of travelling times. Nevertheless, there are many configurations of strictly convex obstacles which are not known to satisfy these conditions. 
In this paper we generalise the uniqueness result of \cite{GNS2023Preprint1}, removing the so-called equivalence up to tangency condition for unions of strictly convex disjoint obstacles on Riemannian manifolds of dimension greater than 2. To do so we make direct use of the results in \cite{GNS2023Preprint1}, as well as adapting some of the arguments made in \cite{MR3359579} in $\reals^m$ $(m>2)$ to our more general Riemannian setting. 

Let $M$ be a complete $m$-dimensional Riemannian manifold ($m\geq 2$). We denote the metric on $M$ by $\langle\ \cdot,\cdot \ \rangle$, and the sectional curvature of $M$ by $sec_M(X,Y)$ for any $X,Y\in TM$. 
We will say that an $m$-dimensional submanifold $M'\subseteq M$ with boundary is strictly convex if at every point of $\partial M'$ the principal curvatures of $\partial M$ are all positive (with respect to the outward unit normal).
Let $S$ be a strictly convex $m$-dimensional submanifold of $M$ with boundary. We shall assume that every pair of points in $S$ can be connected by a unique minimal geodesic in $S$. Denote the unit tangent bundle of $S$ by $\uT{S}$, and the inward unit normal field of $\partial S$ as $N_S$. We define the following subset of inward pointing vectors along the boundary of $S$, 
\[
	\TS = \{\sigma\in \uT{S} : \pr_1(\sigma)\in\partial S \textrm{ and } \langle \sigma, N_S \rangle > 0\}.
\]
Suppose that $K$ is a union\footnote{We ought to note when $d=1$, (i.e. $K$ is a single obstacle) the travelling times of any billiard ray in $S_K$ is finite, and showing the uniqueness of $K$ from $\sT_K$ becomes trivial.} $K = K_1 \cup \dots\cup K_d$ of $d\geq 1$ disjoint $m$-dimensional, strictly convex submanifolds $K_i$ of $S$ with smooth boundary. Let $S_K = \overline{S\backslash K}$, we denote the (smooth) geodesic flow induced by $M$ as $\flow$ and the billiard flow induced by $S_K$ as $\genflow^K$. Denote the unit tangent bundle of $S_K$ by $\uT{S_K}$, and for any $\sigma\in\uT{S_K}$ let $\gamma^+_K(\sigma) = \{\pr_1\circ\genflow_t^K(\sigma):t\geq 0\}$.
We also let $\gamma^+_K(\sigma)(t) = \pr_1\circ\genflow_t^K(\sigma)$. If $\sigma\in\TS$ and there is some $t^*>0$ such that $\gamma_K^+(\sigma)(t^*)\in\partial S$ (i.e. $\gamma_K^+(\sigma)$ intersects the boundary $\partial S$ twice) then we say that $\gamma_K^+(\sigma)$ is \emph{non-trapped}. If no such $t^*$ exists we say that $\gamma_K^+(\sigma)$ is trapped. Denote the set of trapped points $\sigma\in\TS$ which generate a trapped ray $\gamma_K^+(\sigma)$ by $\Trap{\partial S}^{(K)}$.
The set $\sT_K$ of \emph{travelling times} of $K$ is defined to be the following,
\begin{align*}
	\sT_K = \{&(\gamma_K^+(\sigma)(0),\gamma_K^+(\sigma)(t^*),t^*)\in\partial S\times\partial S\times\reals^+:
	\textrm{ for some }\sigma\in\TS\setminus\Trap{\partial S}^{(K)},\\
	&\textrm{ and } t^* \textrm{ is the minimum number } t \textrm{ such that } \gamma_K^+(\sigma)(t)\in\partial S \}.
\end{align*}
Let $L$ also be a union $L = L_1 \cup \dots\cup L_{d'}$ of $d'\geq 1$ disjoint $m$-dimensional, strictly convex submanifolds $L_i$ of $S$ with smooth boundary. Suppose that $K$ and $L$ have the same travelling times, i.e $\sT_K = \sT_L$. Denote the minimum distance between any two distinct components in $K$ by $\mindist^K$ and in $L$ by $\mindist^L$. Set $\mindist = \min\{\mindist^K,\mindist^L\}$, and let $D$ be the diameter of $S$. 

\begin{proposition}[\cite{GNS2023Preprint1}]\label{proposition:convex_front_from_obstacle}
	Fix a point $x_0\in \partial K$ and tangent direction $V\in T_{x_0}\partial K$ such that $\norm{V} = 1$. There is a neighbourhood $U\subseteq \partial K$ of $x_0$ and a strictly convex front $Y$ such that for every $y\in Y$, the ray in the inward normal direction from $y$ will meet $U$ tangentially. Hence, the front $Y$ is diffeomorphic to $U$. Furthermore, there is a global lower bound $\globalcurv_0>0$ such that the minimum principal curvature of $Y$ is greater than $\globalcurv_0$ for all $x_0\in \partial K$ and $V\in T_{x_0}\partial K$. 
\end{proposition}

\begin{lemma}[\cite{GNS2023Preprint1}]\label{lemma:minimal_reflection_angle}
		There exist constants $\minref\in\mathbb{Z}^+$ and $\varphi_0\in (0,\frac{\pi}{2})$ such that any geodesic reflecting transversally on $\partial K$ at least $\minref$ times will hit $\partial K$ at an angle of at most $\varphi_0$ to the outward normal at least once.
\end{lemma}

Denote the maximal sectional curvature of $S$ by $\maxsec$, i.e. $sec_M(X,Y)\leq\maxsec$ for all $X,Y\in\uT{S}$. Let $\mincurv>0$ be a lower bound for the principal curvatures of $\partial K$ and $\partial L$.
We ask that $S$ satisfies one of the following conditions where $\minref$, $\varphi_0$ and $\globalcurv_0$ are defined as in \Cref{proposition:convex_front_from_obstacle} and \Cref{lemma:minimal_reflection_angle} above:
\begin{enumerate}[(A)]
	\item $\maxsec \leq 0$ or,\label{condition:neg_curve}
	\item $\maxsec > 0,$ while $\maxdist\minref\sqrt{\maxsec}<\frac{\pi}{2}$ and $\tan(\maxsec\maxdist\minref)\sqrt{\maxsec}<\globalcurv,$\label{condition:pos_curve}\vspace{0.2cm}\\
	where $\globalcurv = \min\{2\mincurv\cos\varphi_0,\globalcurv_0\}$.
\end{enumerate}
We can now state the main result.

\begin{restatable}{theorem}{mainthm}\label{thm:dim3}
	Suppose that $K$ and $L$ have the same set of travelling times. If $m\geq 3$ and conditions (\ref{condition:neg_curve}) or (\ref{condition:pos_curve}) are satisfied then $K = L$.
\end{restatable}
\noindent
Note that for the case where $m=2$ it is currently not known whether the same holds, even when $M = \reals^2$. \Cref{thm:one-sided} gives a set of sufficient conditions for the result to hold in the two dimensional case, however it is not known whether all obstacles $K$ and $L$ satisfy them. Currently only those pairs of obstacles which satisfy the equivalent up to tangency condition are known to be uniquely determined by their sets of travelling times. The obstruction to the current approach is the trapping set $\Trap{\partial S}$, which may disconnect $\TS$. In fact it is known to do so even in the simple case of two symmetric disks in the plane.

\section{Preliminary Results}
For a hypersurface $\Sigma$ in $S$, given a unit normal vector field $N^*$ we define the shape operator $s_\Sigma$ of $\Sigma$ as $s_\Sigma Y = \nabla_Y N^*$ where $Y$ is any vector tangent to $\Sigma$. We call the eigenvalues of $s_\Sigma$ the principal curvatures of $\Sigma$ with respect to $N^*$.
We say that a hypersurface $\Sigma$ in $S$ is strictly convex if all principal curvatures of $\Sigma$ have the same sign. Let $N_\Sigma(x)$ be the unit normal of $\Sigma$ at $x$ such that the principal curvatures of $\Sigma$ with respect to $N_\Sigma$ are positive. We adopt the convention of calling $N_\Sigma(x)$ the \emph{inward} unit normal to $\Sigma$ at $x$. Similarly we will call $-N_\Sigma(x)$ the \emph{outward} unit normal. Unless otherwise specified, we will always take the principal curvatures of a strictly convex hypersurface in $S$ with respect to $N_\Sigma$.
 
\begin{proposition}[\cite{GNS2023Preprint1}]\label{lemma:always_convex}
	For some $\frontmincurv>0$, the following holds. Suppose that $\Sigma$ is a strictly convex hypersurface in $S$, with principal curvatures bounded below by $\globalcurv>0$. Let 
	\[
		\Sigma_t = \{\pr_1\circ\flow_t(x,N_\Sigma(x)):x\in \Sigma\}.
	\]
	Then the principal curvatures of $\Sigma_t$ are bounded below by $\frontmincurv$, provided\footnote{Note that $\Sigma_t$ may cease to be an embedded submanifold if both $\Sigma$ and $D$ are sufficiently large. One can avoid this issue when working locally by shrinking $\Sigma$ initially.} that $t < \maxdist\minref$.
	
\end{proposition}

\begin{proposition}[\cite{GNS2023Preprint1}]\label{proposition:reflection_sff}
	Suppose that $\Sigma$ is a strictly convex hypersurface in $S$ and $\gamma_{(x,N_\Sigma(x))}$ reflects off an obstacle $\partial K$ at time $t_r$. Let $s^-$ and $s^+$ be the shape operators of $\Sigma$ before and after reflection respectively, at the point of reflection $\gamma_{(x,N_\Sigma(x))}(t_r)\in \partial K$. Let $s_K$ be the shape operator of $\partial K$. Then the shape operators satisfy the following equation,
	\begin{equation}\label{eq:reflection_sff}
		s^+(Y_+) - s^-(Y_-) = -2\langle N_-, N_K\rangle s_K(Y),
	\end{equation}
	for all $Y$ tangent to $K$. Here we denote by $Y_\pm = Y - \langle Y, N_\pm \rangle N_\pm$, where $N_-$ is the normal to $\Sigma$ prior to reflection, $N_+$ is the outward unit normal to $\Sigma$ after reflection and $N_K$ is the normal to $\partial K$.
\end{proposition}

\begin{corollary}[\cite{GNS2023Preprint1}]\label{proposition:theta_convex}
	Suppose that $\Sigma$ is a strictly convex hypersurface in $S$, with principal curvatures bounded below by $\globalcurv>0$. Let $N_\Sigma$ be the outward unit normal field of $X$. Then for any $t>0$, the hypersurface $X_t$ given by
	\[
		\Sigma_t = \{\pr_1\circ\genflow_t(x,N_\Sigma(x)):x\in \Sigma\},
	\]
	is also strictly convex.
\end{corollary} 

\begin{proposition}[\cite{GNS2023Preprint1}]\label{proposition:convex_front_collision}
	Suppose $\Sigma$ and $\Pi$ are two strictly convex fronts such that for some $t_0>0$ and $x\in \Sigma$ we have $x_{t_0}:=\gamma_{(x,N_\Sigma(x))}(t_0)\in \Pi$. Moreover, suppose that $N_{t_0}(x) := \dot\gamma_{(x,N_\Sigma(x))}(t_0)$ is an inward normal to $\Pi$, and that for any principal curvatures $k_\Sigma$ and $k_\Pi$ of $\Sigma$ and $\Pi$ respectively with respect to $N_{t_0}(x)$ at $x_{t_0}$, we have $k_\Pi<-k_{\min}<0<k_{\min}<k_\Sigma.$ Let $\mathcal{G}$ be the set of points $x^*\in \Sigma$ such that $x^*_{t(x^*)}\in \Pi$ and $N_{t(x^*)}(x^*)$ is normal to $\Pi$. Then $\mathcal{G}$ is a submanifold of $\Sigma$ of dimension 0. That is, $\mathcal{G}$ is at most countable.
\end{proposition}

 Given $\sigma\in\uT{S_K}$, let $0=t_0(\sigma)<t_1(\sigma)<t_2(\sigma)<\dots$ be the reflection times of $\gamma^+_K(\sigma)$. That is, $\gamma^+_K(\sigma)(t_i(\sigma))\in\partial K$ is the i-th reflection of $\gamma^+_K(\sigma)$ on $\partial K$. Note that $t_i(\sigma)$ does not necessarily exist for an arbitrary $\sigma\in\uT{S_K}$ and arbitrary integer $i>0$. Given an arbitrary integer $N>0$, fix some $0<\minextra\ll\mindist$, and define $t^N(\sigma) = t_N(\sigma) + \minextra$ if $t_N(\sigma)$ exists. If $t_N(\sigma)$ does not exist let $i^*\geq 0$ be the largest integer such that $t_{i^*}(\sigma)$ exists, and define $t^N(\sigma) = t_{i^*}(\sigma)+\minextra$. Let $\gamma_K^N(\sigma) = \{\pr_1\circ\genflow_t(\sigma):0\leq t \leq t^N(\sigma)\}$.
 We define the following submanifold of $TS$:
\[
		T_{\partial K} S_K = \{(x,\omega)\in\uT{S_K}:x\in\partial K\}.
\]
\begin{proposition}\label{proposition:p_i_general}
	There exists a locally finite, countable family $\{ P_i^{(K)}\}$ of codimension $1$ submanifolds of $\TS\setminus\Trap{\partial S}^{(K)}$ such that $\gamma_K^+(\sigma)$ is not tangent to $\partial K$ for any $\sigma \in \TS\setminus (\cup_i P_i^{(K)}\cup\Trap{\partial S}^{(K)})$.
\end{proposition} 

\begin{proof}
	Consider any component $K_i$ of $K$ and the unit tangent bundle of its boundary $\uT{\partial K_i}$.
	Denote by $\Gamma_i$ the set of vectors $\sigma\in\uT{\partial K_i}\setminus\Trap{S_K}$ such that $\gamma_K^+(\sigma)(t)$ only has transversal reflections with $\partial K$ for all $t>0$. Consider the subset $\Gamma_i^j\subset\Gamma_i$ of vectors $\sigma\in\Gamma_i$ such that $\gamma_K^+(\sigma)$ has exactly $j$ transversal reflections prior to intersecting $\partial S$. Given $\sigma\in\Gamma_i^j$ there is an open neighbourhood $U(\sigma)\subset\Gamma_i^j$ of $\sigma$ and a smooth function $\tau_{\sigma}:U(\sigma)\to\reals$ such that $\gamma_K^+(u)(\tau_{\sigma}(u))\in\partial S$ for all $u\in U(\sigma)$. Let $U(\sigma)$ be the maximal subset for which $\tau_\sigma$ is well-defined and smooth. Note that $\Gamma_i^j$ is disconnected and, in fact, each $U(\sigma)$ will form a component of $\Gamma_i^j$. Since there are finitely many obstacles, and each $\sigma'\in\Gamma_i^j$ generates a trajectory with exactly $j$ reflections, there are finitely many components $U(\sigma_1),\dots,U(\sigma_{n(i,j)})$ in $\Gamma_i^j$. Now for each $1\leq l\leq n(i,j)$ let $\Phi_l:U(\sigma_l)\to\uT{S}$ be the smooth map given by $u\mapsto\genflow_{\tau_{\sigma_l}(u)}(u)$. Note that $\pr_1\circ\Phi_l(u)\in\partial S$ for all $u\in U(\sigma_l)$ by construction. Let $\Psi(x,\omega) = (x,-\omega)$ for all $(x,\omega)\in\uT S$. Then $\Psi\circ\Phi_l:U(\sigma_l)\to\TS$ is a diffeomorphism onto its image. For each $\Gamma_i^j$ denote its components by $\Lambda^{j,l}_i$. As we have shown, $\Phi\circ\Psi_l(\Lambda^{j,l}_i)$ is a codimension 1 submanifold of $\TS\setminus\Trap{\partial S}$. Let \[\{ P_i^{(K)}\} = \bigcup_{i=1}^d\bigcup_{j=0}^\infty\bigcup_{l=1}^{n(i,j)}\{\Phi\circ\Psi_l(\Lambda^{j,l}_i)\}.\]
	
	Suppose that $\xi_0\in\TS\setminus\Trap{\partial S}$ generates a trajectory which is tangent to $\partial K$. Let $j$-th reflection be the first tangent point of $\gamma_K^+(\xi_0)$ for some $j\geq 1$. Let $\partial K_i$ be the component of $\partial K$ such that $\dot\gamma_K^+(\xi_0)(t_j(\xi_0))\in\uT{\partial K_i}$. It follows by our argument above that $\xi_0\in\Psi\circ\Phi_l(\Lambda^{j,l}_i)$ for some $1\leq l\leq n(i,j)$. Therefore every $\sigma \in \TS\setminus (\cup_i P_i^{(K)}\cup\Trap{\partial S})$ generates a ray which is not tangent to $\partial K$. Finally, given $\xi_0\in\TS\setminus\Trap{\partial S}$, the ray $\gamma_K^+(\xi_0)$ has finitely many reflections, denote this number by $r$. Let $V(\xi_0)\subset \TS\setminus\Trap{\partial S}$ be any open neighbourhood of $\xi_0$ such that $\gamma_K^+(\xi_0')$ has at most $r'$ reflections for all $\xi_0'\in V(\xi_0)$ and some $r<r'<\infty$. It follows that $V(\xi_0)$ may intersect at most finitely many submanifolds in $\{ P_i^{(K)}\}$, since 
	\[V(\xi_0)\bigcap\left(\bigcup_{i=1}^d\bigcup_{j=r'+1}^\infty\bigcup_{l=1}^{n(i,j)}\{\Phi\circ\Psi_l(\Lambda^{j,l}_i)\}\right)=\emptyset.\]
	Therefore $\{ P_i^{(K)}\}$ is indeed locally finite, and countable as desired.
\end{proof}

When restricting the billiard rays to at most finitely many reflections $N$, the proof of \Cref{proposition:p_i_general} has the following direct consequence.

\begin{corollary}\label{proposition:p_i}
	Given any integer $N > 0$, there exists a finite family $\{ P_i^{(K,N)}\}$ of codimension $1$ submanifolds of $\TS$ such that $\gamma_K^N(\sigma)$ is not tangent to $\partial K$ for any $\sigma \in \TS\setminus \cup_i P_i^{(K,N)}$.
\end{corollary}

%

\begin{proposition}[\cite{GNS2023Preprint1}]\label{proposition:tangent_twice}
	There exists a countable family $\{\Xi_i\}$ of codimension 2 smooth submanifolds of $\TS$ such that for any $\sigma\in\TS\backslash (\cup_i\Xi_i)$ the billiard ray generated by $\sigma$ is tangent to $\partial K$ at most once.
\end{proposition}

\section{Modifying $N$-admissible Curves}
The purpose of this section is to address special curves which generate billiard rays that contain problematic points, namely points where the boundaries of $K$ and $L$ do not agree on any neighbourhood of said points. We therefore make the following definitions. A point $x\in\partial K$ will be called \emph{regular} if there exists an open neighbourhood $U\subseteq\partial K$ containing $x$ such that $U\cap \partial L = U$. We will call such a neighbourhood $U$ a \emph{regular neighbourhood}. If no regular neighbourhood of $x$ exists we call $x$ an \emph{irregular} point.

A smooth path $\sigma(s)$, $0\leq s\leq a$ (for some $a > 0$), in $\TS\setminus(\cup_i\Xi_i)$ will be called  $N$-{\it admissible} if it has the following properties:

\begin{enumerate}
\item[(a)] $\gamma_K^+(\sigma(0))=\gamma_L^+(\sigma(0))$ is a non-trapped ray with no irregular points.

\item[(b)]  If  $\sigma(s)\in P$ for some $P\in\{P_i^{(K,N)}\}$ and $s\in (0,a]$, then $\sigma$ is transversal to $P$ at $\sigma(s)$ and $\sigma(s) \notin P'$ for any submanifold $P'\in\{P_i^{(K,N)}\}$ such that $P'\neq P$ .
\end{enumerate}

We shall rely on these $N$-admissible paths throughout, hence it is important to first show their existence. It turns out in fact, that every for every $x\in\partial K$ there is some $N$-admissible path $\sigma:[0,a]\to\TTS$ such that the ray generated by the endpoint $\sigma(a)$ of the path will contain $x$. This will follow by \Cref{lemma:always_reachable} and the following \namecref{proposition:N_admissible}.

\begin{proposition}\label{proposition:N_admissible}
	For every $\sigma^*\in\TS\setminus(\cup_i\Xi_i)$ there exists an $N$-admissible path $\sigma:[0,a]\to\TS\setminus(\cup_i\Xi_i)$ such that $\sigma(a)=\sigma^*$.
\end{proposition}
\begin{proof}
	By \Cref{proposition:tangent_twice} every $\sigma\in\TS\setminus(\cup_i\Xi_i)$ will generate a billiard ray that is tangent to $\partial K$ at most once. Since the submanifolds $\Xi_i$ have codimension 2 in $\TS$ it follows that $\TS\setminus(\cup_i\Xi_i)$ is path-connected. Pick $\sigma_0\in\TS\setminus(\cup_i\Xi_i)$ such that $\gamma_K^+(\sigma_0)$ is a free ray in $S_K$, i.e. a smooth geodesic that does not intersect $\partial K$. Now given $\sigma^*\in\TS\setminus(\cup_i\Xi_i)$ let $\widetilde\sigma:[0,a]\to\TS\setminus(\cup_i\Xi_i)$ be a smooth path such that $\widetilde\sigma(0) = \sigma_0$ and $\widetilde\sigma(a)=\sigma^*$. Denote the set of all such $\widetilde\sigma$ by $P(\sigma_0,\sigma^*)$ which we consider as a subspace of $C^\infty([0,a],\TS\setminus(\cup_i\Xi_i))$ with the Whitney Topology.
	By \Cref{proposition:p_i} there are finitely many submanifolds $P_i$ to which our desired curve $\sigma$ must be transversal. Thus by the transversality theorem 
	(\cite{MR1336822}) the subset $P_R(\sigma_0,\sigma^*)$ of paths in $P(\sigma_0,\sigma^*)$ which are transversal to $P_i$ for all $i$ are residual in $P(\sigma_0,\sigma^*)$. Hence $P_R(\sigma_0,\sigma^*)$ is dense in $P(\sigma_0,\sigma^*)$ by the Baire category theorem. We may now pick any $\sigma\in P_R(\sigma_0,\sigma^*)$ such that $\sigma$ intersects at most one submanifold $P_i$ for each $s\in[0,a]$. Thus we have an $N$-admissible path $\sigma:[0,a]\to\TS\setminus(\cup_i\Xi_i)$ such that $\sigma(a)=\sigma^*$.
\end{proof} 

Eliminating irregular points on a single billiard ray $\gamma_K^+(\sigma_*)$ (for some $\sigma_*\in\TS$) can be prohibitively difficult. Particularly if $\sigma_*\in\Trap{\partial S}^{(K)}$, i.e. $\gamma_K^+(\sigma_*)$ is trapped, where there can be infinitely many such points along the ray. However, when there is an $N$-admissible curve $\sigma$ passing through $\sigma_0$ the difficulties are addressed on two fronts. Firstly, in this section we need only consider the first $N$ reflections of $\gamma_K^+(\sigma_*)$, thus not requiring us to worry about infinitely many irregular points along one billiard ray. Secondly requiring that $\gamma_K^+(\sigma(0))$ has no irregular points will ensure that for $s$ sufficiently close to $0$, the rays $\gamma_K^N(\sigma(s))$ will have no irregular points either. Therefore it will be sufficient to consider $N$-admissible curves $\sigma:[0,a]\to\TTS$ which generate rays that have no irregular points, except for $\sigma(a)$ which will generate a ray $\gamma_K^N(\sigma(a))$ that contains some irregular points. Our aim is to, in effect,  reduce the number of irregular points on this final ray down to one, at which point we may use the travelling time functions of $K$ and $L$ to compare the rays $\gamma_K^+(\sigma(a))$ and $\gamma_L^+(\sigma(a))$. To do so we have to consider a few cases, each of which will be covered by one of \Cref{lemma:N-non_tangent,lemma:N-tangent_irregular,lemma:N-tangent_regular}.

The first case, covered by \Cref{lemma:N-non_tangent}, assumes that the troublesome ray $\gamma^N_K(\sigma(a))$ is not tangent to $\partial K$. Since $\gamma^N_K(\sigma(a))$ only has transversal reflections, it follows that the same holds for $\gamma^N_K(\sigma(s))$ with $s<a$ sufficiently close to $a$. Note that by assumption these rays, $\gamma^N_K(\sigma(s))$, will have no irregular points. Hence we shall modify $\sigma$ near $a$ by perturbing $\gamma^N_K(\sigma(a))$ from the first irregular point, such that the resulting ray has a regular point in place of the second irregular point. This is only possible owing to the rays $\gamma^N_K(\sigma(s))$ having to irregular points for all $s<a$.

\begin{lemma}\label{lemma:N-non_tangent}
	Let $\sigma:[0,a] \to \TTS$ be an $N$-admissible path such that:
	\begin{enumerate}[(i)]
		\item $\gamma_K^N(\sigma(a))$ is not tangent to $\partial K$.
		\item $\gamma_K^N(\sigma(s))$ contains no irregular points for all $s\in [0,a)$.
		\item $\gamma_K^N(\sigma(a))$ contains exactly $n$ irregular points (for some integer $1<n\leq N$).
	\end{enumerate}
	Then for some sufficiently small $\delta > 0$ there exists an $N$-admissible path $\sigma^*:[0,a]\to\TTS$ such that:
	\begin{enumerate}[(i)]
		\item $\gamma_K^N(\sigma^*(a))$ is not tangent to $\partial K$.
		\item $\gamma_K^N(\sigma^*(s))$ contains no irregular points for all $s\in [0,a-3\delta)$.
		\item $\gamma_K^N(\sigma^*(s))$ contains at most $n-1$ irregular points and has the same number of reflections as $\gamma_K^N(\sigma(a))$ for all $s\in [a-3\delta,a]$.
		\item $\gamma_K^N(\sigma^*(a))$ contains the first irregular point of $\gamma_K^N(\sigma(a))$.
	\end{enumerate}
\end{lemma}
\begin{proof}
	Let $x_1,\dots,x_n$ (in order of appearance) be the irregular points of $\gamma = \gamma_K^N(\sigma(a))$. Note that for each $x_i$ there is an open subset $U_i\subseteq\partial K$ such that $x_i\in\partial U_i$ and $\partial K \cap U_i = \partial L\cap U_i$. Indeed, suppose $x_i$ is the $j$-th reflection point of $\gamma$. Then for $s<a$ sufficiently close, since $x_i$ is a transversal point of reflection, the rays $\gamma_K^N(\sigma(s))$ will all have their $j$-th reflection $x_i(s)$ within a small neighbourhood $V_i\subset\partial K$ of $x_i$. Now since each $x_i(s)$ is a regular point, there is some $V_i(s)\subseteq V_i$ such that $V_i(s)\cap\partial L = V_i(s)$. Let
	\[
		U_i = \bigcup_{s\in [0,a):x_i(s)\in V_i}{V_i(s)}.
	\]
	Clearly $U_i$ is open, $U_i\cap\partial L = U_i$ and $x_i\in\partial U_i$ as claimed.

	\begin{figure}
		\center
		\begin{tikzpicture}
			\draw (-6,0) arc (90:-45:2);
			\coordinate (k1) at ($(-6,-2)+(-45:2)+(-0.5,0.5)$);
			\node at (k1) {$\partial K_{j_1}$};
			\coordinate (x1) at ($(15:2) + (-6,-2)$);
			\draw[blue,thick, -{Parenthesis}] (x1) arc (15:-40:2) node[midway,left] {$U_1$};
			\node[left] at (x1) {$x_1$};
			\fill (x1) circle(0.05cm);
			\draw[dashed] (x1) -- +(25:1) coordinate (X1);
			\draw ($(x1)+(30:1)$) -- +(25:-0.1) -- ($(x1)+(25:0.9)$);
			\draw[dashed] (x1) -- ($(x1)+(-2:1)$) coordinate (mx1);
			\draw ($(x1)+(3:1)$) -- +(-2:-0.1) -- ($(x1)+(-2:0.9)$);
			\draw[dashed] (x1) -- ($(x1)+(-30:1)$);
			\draw ($(x1)+(-25:1)$) -- +(-30:-0.1) -- ($(x1)+(-30:0.9)$);
			\draw[thick,red, <->] (x1)+(40:1) arc (40:-45:1);
			\node[red, right, above] at ($(X1) + (-0.1,0.3)$) {$\Sigma$};
			
			\draw (6,0) arc (90:245:2);
			\coordinate (k2) at ($(6,-2)+(245:2)+(0.5,0.5)$);
			\node at (k2) {$\partial K_{j_2}$};
			\coordinate (x2) at ($(165:2) + (6,-2)$);
			\draw[blue,thick, -{Parenthesis}] (x2) arc (165:230:2) node[midway,right] {$U_2$};
			\node[right] at (x2) {$x_2$};
			\fill (x2) circle(0.05cm);
			
			\coordinate (Y1) at ($(x2)+(-30:-2)$);
			\coordinate (y1) at ($(Y1) + (20:-1) + (-12:1)$);
			\draw[thick,red, <->] (Y1) arc (20:-65:1) node[red, right, below] {$\Pi$};
			\fill (y1) circle(0.05cm);
			\node at ($(y1)+(-0.2,0.2)$) {$p_1$};
			\draw[dashed] (x2) -- (y1);
			\draw ($(Y1) + (20:-1) + (-7:1)$) -- + (-12:0.1) -- ($(Y1) + (20:-1) + (-12:1.1)$);
			\coordinate (y2) at ($(Y1) + (20:-1) + (-20:1)$);
			\fill (y2) circle(0.05cm);
			\coordinate  (labely2) at ($(y2) + (0.2,-0.7)$);
			\draw[bend right, ->, shorten >= 0.1cm] (labely2) node[below] {$p^*_1$} to (y2);
			\draw[dashed] (y2) -- ($(175:2) + (6,-2)$) coordinate (xs2);
			\draw ($(Y1) + (20:-1) + (-15:1)$) -- + (-20:0.1) -- ($(Y1) + (20:-1) + (-20:1.1)$);
			\fill (xs2) circle(0.05cm) node[right] {$x_2(a-\delta)$};
			
			\draw [->, shorten <= 1cm, shorten >= 1cm, bend left] (mx1) to (y2) node[midway] {$\pr_1\circ\genflow_{\tau}$}; 
		\end{tikzpicture}
		\caption{Construction of $\Sigma$, $\Pi$ and $p^*_1$.}\label{fig1}
	\end{figure}
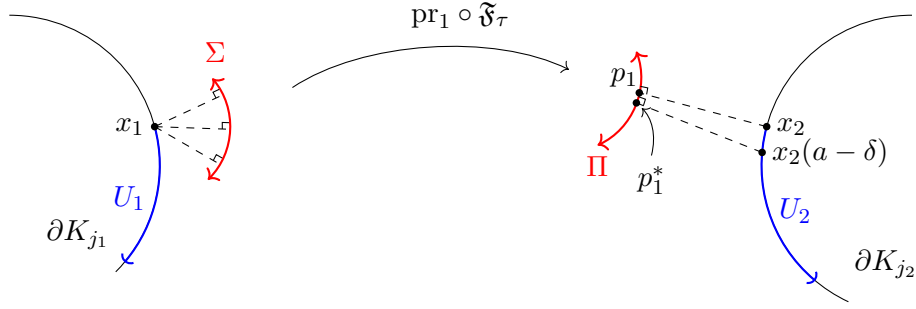

	Let $\omega_i\in\uT{S}$ be the reflected direction of $\gamma$ at $x_i$, for $i=1,\dots,n$. Then there are $\tau_i>0$ such that $\genflow_{\tau_i}^K(x_i,\omega_i) = (x_{i+1},\omega_{i+1})$ for all $i=1,\dots,n-1$. Pick some small $\delta,\varepsilon>0$ and let
	\[
		\Sigma = \{\exp_{x_1}(\varepsilon\omega):\omega\in\uT{S_{x_1}},\norm{\omega-\omega_1}<\delta\}.
	\]
	By \cite{MR618545}, the $m-1$-dimensional submanifold $\Sigma$ will have a minimum principal curvature greater than $\Theta$ provided that $\varepsilon>0$ is sufficiently small. We fix $\varepsilon$, and note that although we may need to shrink $\delta$ later, that will not affect the minimum principal curvature of $\Sigma$. Also set
	\[
		\widetilde\Sigma = \{(x,N_\Sigma(x)):x\in\Sigma\textrm{ and }N_\Sigma(x)\textrm{ is the outward unit normal to } \Sigma\textrm{ at }x\}.
	\]
	Now, similar to $x_i(s)$, let $\omega_i(s)$ be the reflected direction of $\gamma_K^N(\sigma(s))$ at $x_i(s)$ for $i=1,\dots,n$. Fix $\tau < \tau_1$ close to $\tau_1$ such that the hypersurface $\Pi = \pr_1(\widetilde\Pi)$, where $\widetilde\Pi = \genflow_\tau^K(\widetilde\Sigma)$, does not intersect $\partial K$. Note that we may also shrink $\delta$ to ensure $\Pi$ does not intersect $\partial K$ (\Cref{fig1}). For every $p\in\Pi$ let $N_\Pi(p)$ be the outward unit normal to $\Pi$ at $p$. Fix a small $\lambda>0$ and set
	\[
		\widehat\Pi = \{(p,\omega):p\in\Pi,\omega\in\uT{S}_p,\norm{\omega-N_\Pi(p)}<\lambda\}.
	\]
	Recall that we defined
	\[
		T_{\partial K} S_K = \{(x,\omega)\in\uT{S_K}:x\in\partial K\}.
	\]
	Provided $\delta$ is sufficiently small, there is an open neighbourhood $V\subseteq T_{\partial K}S_K$ of $(x_1,\omega_1)$ and there exists a smooth function $t:V\to\reals$ such that ${\genflow_{t(x,\omega)}^K(x,\omega)\in\widehat\Pi}$. Let $\Phi:V\to\widehat\Pi$ be the diffeomorphism defined by $\Phi(x,\omega) = \genflow_{t(x,\omega)}^K(x,\omega)$. We may shrink $\delta$ so that $(x_1(s),\omega_1(s))\in V$ for all $[a-3\delta,a]$. Now since $\Sigma$ is a strictly convex hypersurface with minimum principal curvature greater than $\Theta$, by \Cref{proposition:theta_convex} it follows that $\Pi$ is also strictly convex. Thus for any $s\in[0,a]$ sufficiently close to $a$, there exists a $p\in\Pi$ such that $\exp_p(\widetilde t(p)N_\Pi(p)) = x_2(s)$, for some $\widetilde t(p) > 0$ close to $\tau_1 - \tau$. Possibly shrinking $\delta$ once again, we may assume that $x_2(s)\in U_2$ for all $s\in[a-3\delta,a]$. Let $p_1^*\in\Pi$ be such that $\exp_{p_1^*}(\widetilde t(p_1^*)N_\Pi(p_1^*)) = x_2(a-\delta)$  as in \Cref{fig1}. Then take a smooth curve $p^*:[a-3\delta,a]\to\Pi$ such that $p^*(s) = \pr_1\circ\Phi(x_1(s),\omega_1(s))$ for all $s\in[a-3\delta,a-\delta]$ and $p^*(a)=p^*_1$. Let $\psi:[0,a]\to[0,1]$ be a smooth function such that $\psi(s) = 0$ for all $s\leq a - 3\delta$, and $\psi(s) = 1$ for all $s\geq a-2\delta$ (see e.g. \cite{LEESMOOTH}). Now define,
	\[
		s^*(s) = \psi(s) \left( \frac{a-s}{2} - \delta \right) + s.
	\]
	Then $s^*(s) = s$ for all $s\in [a-3\delta,a-2\delta]$, while $s^*(s)\in(a-2\delta,a-\delta]$ for all $s\in ( a-2\delta,a]$. In particular, note that $s^*(a) = a-\delta$.
	
	Now let $v:[a-3\delta,a]\to\uT{S_K}$ be the smooth function such that, for each $s\in[a-3\delta,a]$, we have $\exp_{p^*(s)}(v(s)\widehat t(s)) = x_2(s^*(s))$ for some $\widehat t(s)$ close to $\tau_1 - \tau$. It follows that $(p^*(s),v(s)) = \Phi(x_1(s),\omega_1(s))$ for all $s\in[a-3\delta,a-2\delta]$ since $p^*(s) = \pr_1\circ\Phi(x_1(s),\omega_1(s))$ and $s^*(s) = s$ whenever $s\in[a-3\delta,a-2\delta]$. Furthermore, when $s = a$ we have $p^*(a) = p_1^*$ and $s^*(a) = a-\delta$ by construction. Thus $\exp_{p^*(a)}(v(a)\widehat t(a)) = x_2(a-\delta)$, that is, $\widehat t(a) = \widetilde t(p^*_1)$ and more importantly $v(a) = N_\Pi(p_1^*)$. Therefore $(p^*(a),v(a))\in\widetilde\Pi$, i.e. \[\pr_1\circ\Phi^{-1}(p^*(a),v(a)) = x_1.\]
	It is important to note here that $\pr_1\circ\Phi^{-1}(\widetilde\Pi) = \{x_1\}$, while $\pr_1\circ\Phi^{-1}(\widehat\Pi) = \pr_1(V) \supset \{x_1\}$. Now set $(x^*_1(s),\omega^*_1(s)) = \Phi^{-1}(p^*(s),v(s))$ for all $s\in[a-3\delta,a]$. Then for $s\in[a-3\delta,a-2\delta]$ we get $(x^*_1(s),\omega^*_1(s)) = (x_1(s),\omega_1(s))$. Moreover, for $s\in [a-2\delta,a]$ we have $\Phi(x^*_1(s),\omega^*_1(s)) = (p^*(s),v(s))$, so that the ray issued from $x^*_1(s)$ in the direction $\omega^*_1(s)$ will hit $\partial K$ at $x_2(s^*(s))\in U_2$ after sufficiently many reflections.
	
		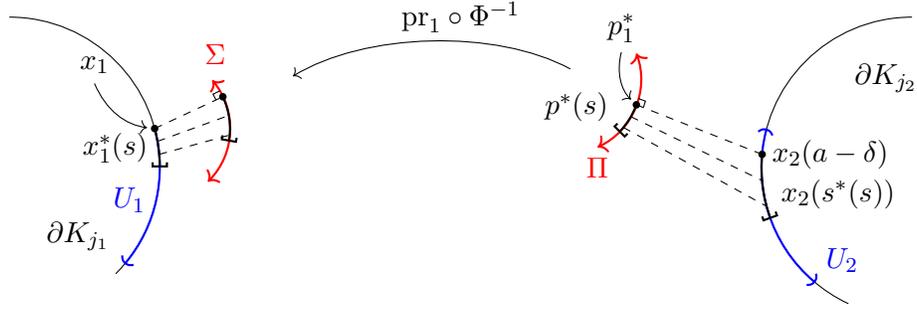
\begin{figure}
		\center
		\begin{tikzpicture}
			\draw (-6,0) arc (90:-45:2);
			\coordinate (k1) at ($(-6,-2)+(-45:2)+(-0.5,0.5)$);
			\node at (k1) {$\partial K_{j_1}$};
			\coordinate (x1) at ($(15:2) + (-6,-2)$);
			\draw[blue,thick, -{Parenthesis}] (x1) arc (15:-40:2) node[midway,left] {$U_1$};
			\draw[thick, -{Bracket}] (x1) arc (15:0:2) node[midway,left] {$x_1^*(s)$};
			\fill (x1) circle(0.05cm);
			\coordinate  (labelx1) at ($(x1) + (-0.8,0.6)$);
			\draw[bend right, ->, shorten >= 0.1cm] (labelx1) node[above] {$x_1$} to (x1);
			
			\draw[thick,red, <->] (x1)+(40:1) arc (40:-45:1);
			\draw[dashed] (x1) -- +(25:1) coordinate (X1);
			\fill (X1) circle(0.05cm);
			\draw[thick, -{Bracket}] (X1) arc (25:-10:1);
			\draw ($(x1)+(30:1)$) -- +(25:-0.1) -- ($(x1)+(25:0.9)$);
			\draw[dashed] ($(10:2) + (-6,-2)$) -- ($(x1)+(10:1)$) coordinate (mx1);
			\draw[dashed] ($(5:2) + (-6,-2)$) -- ($(x1)+(-5:1)$);
			\node[red, right, above] at ($(X1) + (-0.1,0.3)$) {$\Sigma$};
			
			\draw (6,0) arc (90:245:2);
			\coordinate (k2) at ($(6,-2)+(245:2)+(0.5,3)$);
			\node at (k2) {$\partial K_{j_2}$};
			\coordinate (x2) at ($(165:2) + (6,-2)$);
			\draw[blue,thick, {Parenthesis}-{Parenthesis}] (x2) arc (165:230:2) node[anchor = south west] {$U_2$};
			
			\coordinate (Y1) at ($(x2)+(-30:-2)$);
			\coordinate (y1) at ($(Y1) + (20:-1) + (-12:1)$);
			\draw[thick,red, <->] (Y1) arc (20:-65:1) node[red, right, below] {$\Pi$};
			\coordinate (y2) at ($(Y1) + (20:-1) + (-20:1)$);
			\fill (y2) circle(0.05cm);
			\draw[thick, -{Bracket}] (y2) arc (-20:-45:1) node[anchor = south east] {$p^*(s)$};
			\coordinate  (labely2) at ($(y2) + (-0.2,0.7)$);
			\draw[bend right, ->, shorten >= 0.1cm] (labely2) node[above] {$p^*_1$} to (y2);
			\draw[dashed] (y2) -- ($(175:2) + (6,-2)$) coordinate (xs2);
			\draw[dashed] ($(Y1) + (20:-1) + (-30:1)$) -- ($(185:2) + (6,-2)$);
			\draw[dashed] ($(Y1) + (20:-1) + (-40:1)$) -- ($(195:2) + (6,-2)$);
			\draw ($(Y1) + (20:-1) + (-15:1)$) -- + (-20:0.1) -- ($(Y1) + (20:-1) + (-20:1.1)$);
			\fill (xs2) circle(0.05cm) node[right] {$x_2(a-\delta)$};
			
			\coordinate (xa2) at ($(200:2) + (6,-2)$);
			\draw[thick, -{Bracket}] (xs2) arc (175:200:2) node[anchor = south west] {$x_2(s^*(s))$};
			
			\draw [->, shorten <= 1cm, shorten >= 1cm, bend right] (y2) to (mx1) node[midway] {$\pr_1\circ\Phi^{-1}$}; 
		\end{tikzpicture}
		\caption{Construction of $x^*_1(s)$ and $p^*(s)$.}\label{fig2}
	\end{figure}
	
	We will now use our smooth curve $p^*$ to define the smooth path $\sigma^*:[0,a]\to\TTS$. For $s\in[0,a-3\delta]$ let $\sigma^*(s) = \sigma(s)$. For each $s\in[a-3\delta,a]$ let $\sigma^*(s)$ be the unique vector in $\TTS$ such that
	\[
		\genflow_{t^*(s)}(\sigma^*(s)) = (x^*_1(s),\omega^*_1(s)),
	\]
	for some $t^*(s)>0$. Provided $\delta$ is sufficiently small, $t^*(s)$ is smooth, thus $\sigma^*$ is smooth as well. We may also assume that $\gamma_K^N(\sigma^*(s))$ has no tangent points for all $s\in[a-3\delta,a]$ by shrinking $\delta$ if necessary. Furthermore, $\gamma_K^N(\sigma^*(a))$ contains $x_1$, while $x_2$ is not contained in $\gamma_K^N(\sigma^*(s))$ for all $s\in[0,a]$ by construction. Suppose that $z_i$ is a regular point of $\gamma$, then there is a neighbourhood $Z_i\subseteq\partial K$ of $z_i$ such that $Z_i\cap\partial L = Z_i$. By construction, $\gamma_K^N(\sigma^*(s))\cap Z_i\neq\emptyset$, i.e. $\gamma_K^N(\sigma^*(s))$ must have a transversal reflection with $\partial K$ near $z_i$ (within the regular neighbourhood $Z_i$). 
	Hence there can be at most $n$ irregular points on $\gamma_K^N(\sigma^*(s))$ for all $s\in [a-3\delta,a]$, and since $\gamma_K^N(\sigma^*(s))$ passes through $x_2(s^*(s))\in U_2$ there can in fact be at most $n-1$ irregular points. Thus we have constructed the $N$-admissible curve $\sigma^*$ with the desired properties.
\end{proof}

The purpose of the following lemma is to reduce the case where the final ray $\gamma_K^N(\sigma(a))$ has a tangent point of reflection at a regular point to the prior case as in \Cref{lemma:N-non_tangent}. In this case we do not need to eliminate an irregular point, but instead we only have to remove the tangency in the final ray. We do so by carefully perturbing the ray $\gamma_K^N(\sigma(a))$, while ensuring that the resulting ray remains within the regular neighbourhoods that $\gamma_K^N(\sigma(a))$ intersects.

\begin{lemma}\label{lemma:N-tangent_regular}
		Let $\sigma:[0,a] \to \TTS$ be an $N$-admissible path such that:
	\begin{enumerate}[(i)]
		\item $\gamma_K^N(\sigma(a))$ is tangent to $\partial K$ at some regular point $y_0\in\partial K_{i^*}$.
		\item $\gamma_K^N(\sigma(s))$ contains no irregular points for all $s\in [0,a)$.
		\item $\gamma_K^N(\sigma(a))$ contains exactly $n$ irregular points (for some integer $1<n\leq N$).
	\end{enumerate}
	Then for some sufficiently small $\delta^* > 0$ there exists an $N$-admissible path $\sigma^*:[0,a]\to\TTS$ such that:
	\begin{enumerate}[(i)]
		\item $\gamma_K^N(\sigma^*(a))$ is not tangent to $\partial K$.
		\item $\gamma_K^N(\sigma^*(s))$ contains no irregular points for all $s\in [0,a-3\delta^*)$.
		\item $\gamma_K^N(\sigma^*(s))$ contains at most $n$ irregular points and has at most the same number of reflections as $\gamma_K^N(\sigma(a))$ for all $s\in [a-3\delta^*,a]$.
		\item $\gamma_K^N(\sigma^*(a))$ contains the first irregular point of $\gamma_K^N(\sigma(a))$.
	\end{enumerate}
\end{lemma}
\begin{proof}
	Since $y_0$ is a regular point, there is a neighbourhood $W\subseteq\partial K_{i^*}$ of $y_0$ such that $W = W\cap\partial L$. Set $\gamma = \gamma_K^N(\sigma(a))$, and let $x_1$ be the first irregular point of $\gamma$. Note that since $\sigma$ is $N$-admissible, $\gamma$ can have at most one tangent point and hence $x_1$ is a transversal point of reflection. Let $\omega_1$ be the reflected direction of $\gamma$ at $x_1$. If $y_0$ occurs prior to $x_1$ along $\gamma$ then we replace $\omega_1$ with its reflection on $\partial K$ at $x_1$ and proceed in the same manner. Set $\varepsilon,\delta > 0$ sufficiently small so that the minimum principal curvature of the submanifold
		\[
			\Sigma = \{\exp_{x_1}(\varepsilon\omega):\omega\in\uT{S_{x_1}},\norm{\omega-\omega_1}<\delta\}
		\]
		is greater than $\Theta$ (\cite{MR618545}). Let $N_\Sigma(x)$ be the outward unit normal to $\Sigma$ at $x$, and set $\widetilde\Sigma = \{(x,N_\Sigma(x)):x\in\Sigma\}$. Let $\tau>0$ be such that $\genflow_\tau^K(x_1,\omega_1)=(y_0,u_0)$, where $u_0$ is the reflected direction of $\gamma$ at $y_0$. Then for some $0<\tau'<\tau$ close to $\tau$ set $\widetilde\Pi = \genflow_{\tau'}^K(\widetilde\Sigma)$, and $\Pi = \pr_1(\widetilde\Pi)$. Then by \Cref{proposition:theta_convex}, $\Pi$ is a strictly convex hypersurface in $S$. Note that we may pick $\tau'$ (while possibly shrinking $\delta$) such that $\Pi\cap\partial K = \emptyset$ and such that there is a neighbourhood $\Omega\subseteq S$ of $y_0$ such that $\Pi\subseteq \Omega$ and $\partial K_i\cap\Omega = \emptyset$ for all $i\neq i^*$. Now let $\alpha_0\in\widetilde\Pi$ be the point such that $\genflow_{\tau-\tau'}(\alpha_0) = (y_0,u_0)$. 
		
		
		Recall that we denote the reflection times of $\gamma^+_K(\sigma(a))$ by $0=t_0(\sigma(a))<t_1(\sigma(a))<t_2(\sigma(a))<\dots$. Let $\rho<N$ be such that $\gamma_K^+(\sigma(a))(t_{\rho}(\sigma(a))) = y_0$. i.e., $y_0$ is the $\rho$-th reflection of $\gamma_K^+(\sigma(a))$. 
		Since $\gamma$ has no tangential reflections except for $y_0$, then possibly shrinking $\delta$, for each $\alpha\in\widetilde\Pi$ there is a corresponding $\tilde\sigma(\alpha)\in\TTS$, and a time $\tau(\alpha)\in\reals$ such that $\genflow_{\tau(\alpha)}(\alpha)=\tilde\sigma(\alpha)$. Denote this mapping as $\Phi(\alpha) = \genflow_{\tau(\alpha)}(\alpha)$, for all $\alpha\in\widetilde\Pi$. Then $\Phi$ is a diffeomorphism onto the $n-1$ dimensional submanifold $\Phi(\widetilde\Pi)\subseteq\TTS$. 
		
		Note that the set of points $\alpha\in\widetilde\Pi$ such that $\gamma_K^+(\alpha)$ is tangent to $\partial K$ is covered by a countable union of codimension 1 submanifolds. We can therefore pick a point $\alpha^*\in\widetilde\Pi$ such that $\gamma_K^+(\alpha^*)$ is not tangent to $\partial K$. Furthermore, since $\gamma$ is tangent to $\partial K$, by \Cref{proposition:p_i} there is some $P\in\{P_i^{(K,N)}\}$ such that $\sigma(a)\in P$. Since $\sigma$ is $N$-admissible, it follows by definition that $\sigma$ is transversal to $P$ at $\sigma(a)$ and $\sigma(a)\not\in P'$ for all $P'\in\{P_i^{(K,N)}\}\backslash\{P\}$. Thus we may find $\alpha^*$ sufficiently close to $\alpha_0$ such that $\gamma_K^N(\Phi^{-1}(\alpha^*))$ is not tangent to $\partial K$ and $\gamma_K^N(\Phi^{-1}(\alpha^*))$ will also intersect $\partial K$ transversally within any regular neighbourhood that $\gamma$ intersects.

	Moreover, picking $0<s^*<a$ sufficiently close to $a$, it follows that $\gamma_K^N(\sigma(s^*))$ is not tangent to $\partial K$ and has no irregular points (by assumption). Note that both $\gamma_K^N(\sigma(s^*))$ and $\gamma_K^N(\Phi^{-1}(\alpha^*))$ transversally intersect $\partial K$ within every regular neighbourhood of $\gamma$. Therefore we can smoothly extend $\sigma|_{[0,s^*]}$ to $\sigma^*:[0,a]\to\TTS$ such that $\sigma^*|_{[0,s^*]}=\sigma|_{[0,s^*]}$, while $\sigma^*(a) = \Phi^{-1}(\alpha^*)$ and $\gamma_K^N(\sigma^*(s))$ transversally intersects $\partial K$ within every regular neighbourhood of $\gamma$ for all $s\in[s^*,a]$. Thus $\gamma_K^N(\sigma^*(s))$ has at most $n$ irregular points for all $s\in[s^*,a]$.
	By \Cref{proposition:p_i}, \Cref{proposition:tangent_twice} and Thom's transversality theorem, we may assume that $\sigma^*$ satisfies condition (b) of $N$-admissible paths. 
	It follows that $\sigma^*(s)$ is $N$-admissible, and $\gamma_K^N(\sigma^*(a))$ contains $x_1$ (by construction) along with at most $n-1$ other irregular points.
\end{proof}

\Cref{lemma:N-tangent_irregular} is the most technical of the three lemmas in this section. It addresses the case where the final ray $\gamma_K^N(\sigma(a))$ is tangent to $\partial K$ at an irregular point. Since this case cannot be reduced to the first case, as in \Cref{lemma:N-tangent_regular}, we have to eliminate the irregular tangent point directly. The difficulty presents itself in having to work around the tangent irregular point, where unlike the prior case, there is no regular neighbourhood to remain within. When perturbing the ray $\gamma_K^N(\sigma(a))$ we then have two options, the first is "lifting" the ray off the irregular tangent point, and thereby reducing the number of reflections, as well as irregular points. One has to take care in doing so to ensure that no new tangent points are created, as well as to remain within the regular neighbourhoods that $\gamma_K^N(\sigma(a))$ intersects. The other option is to translate the tangent point of reflection along $\partial K$ into a nearby regular neighbourhood which is ensured to exist due to the requirements of the $N$-admissible path $\sigma$. Once again we must remain vigilant to not introduce any new tangent points or irregular points along the new ray.

\begin{lemma}\label{lemma:N-tangent_irregular}
		Let $\sigma:[0,a] \to \TTS$ be an $N$-admissible path such that:
	\begin{enumerate}[(i)]
		\item $\gamma_K^N(\sigma(a))$ is tangent to $\partial K$ at some irregular point $y_0\in\partial K_{i^*}$.
		\item $\gamma_K^N(\sigma(s))$ contains no irregular points for all $s\in [0,a)$.
		\item $\gamma_K^N(\sigma(a))$ contains exactly $n$ irregular points (for some integer $1<n\leq N$).
		\item $y_0$ is not the first irregular point of $\gamma_K^N(\sigma(a))$.
	\end{enumerate}
	Then for some sufficiently small $\delta > 0$ there exists an $N$-admissible path $\sigma^*:[0,a]\to\TTS$ such that:
	\begin{enumerate}[(i)]
		\item $\gamma_K^N(\sigma^*(a))$ is not tangent to $\partial K$.
		\item $\gamma_K^N(\sigma^*(s))$ contains no irregular points for all $s\in [0,a-3\delta)$.
		\item $\gamma_K^N(\sigma^*(s))$ contains at most $n-1$ irregular points and has at most the same number of reflections as $\gamma_K^N(\sigma(a))$ for all $s\in [a-3\delta,a]$.
		\item $\gamma_K^N(\sigma^*(a))$ contains the first irregular point of $\gamma_K^N(\sigma(a))$.
	\end{enumerate}
\end{lemma}
\begin{proof}
	Since $\gamma_K^N(\sigma(a))$ is tangent to $\partial K$ at $y_0$, it follows that $\sigma(a)\in P_\lambda^{(K,N)}$ for some $\lambda$ (\Cref{proposition:p_i}). By the definition of $N$-admissible paths, $\sigma(a)\not\in P_{\lambda'}^{(K,N)}$ for all $\lambda'\neq\lambda$. Thus we can pick a path-connected neighbourhood $W_0\subseteq\TTS$ of $\sigma(a)$ such that $W_0\cap P_{\lambda'}^{(K,N)} = \emptyset$ for all $\lambda'\neq\lambda$. We may shrink $W_0$ around $\sigma(a)$ such that $W_0\setminus P_\lambda^{(K,N)}$ has exactly two path components, since $P_\lambda^{(K,N)}$ has codimension 1 in $W_0$. Pick $\delta>0$ sufficiently small such that $\sigma([a-3\delta,a))$ is contained entirely within one component of $W_0\setminus P_\lambda^{(K,N)}$, and denote this component by $W_0^*$. Set $\gamma = \gamma_K^N(\sigma(a))$. For each irregular point $x_1,\dots,x_n$ (in order of appearance along $\gamma$) denote the reflected direction of $\gamma$ at $x_j$ by $\omega_j$. Set $1<i\leq n$ such that $x_i = y_0$ and denote $u_0 = \omega_i$. Let $\tau>0$ be such that $\genflow_\tau(x_1,\omega_1)=(y_0,u_0)$. That is, $\tau$ is the time $\gamma$ requires to travel between the first irregular point and the tangential irregular point. For some small $\varepsilon,\widetilde\delta>0$, let
	\[
		X = \{\exp_{x_1}(\varepsilon v) : \norm{v-\omega_1}<\widetilde\delta, v\in\uT{S}_{x_1} \}.
	\]
	We choose $\varepsilon$ sufficiently small so that the minimum principal curvature of $X$ is greater than $\globalcurv$ (\cite{MR618545}).
	Let $\nu(x)$ be the outward unit normal to $X$ at $x$, and denote $\widetilde X = \{(x,\nu(x)):x\in X\}$. Pick $\tau^*<\tau-\varepsilon$ close to $\tau-\varepsilon$, and consider $\widetilde Y = \genflow_{\tau^*}(\widetilde X)$. Set $Y=\pr_1(\widetilde Y)$, and $y_0^* = \pr_1\circ\genflow_{\tau^*+\varepsilon}(x_1,\omega_1)\in Y$. By \Cref{proposition:theta_convex}, $Y$ is a strictly convex $n-1$ dimensional submanifold of $S_K$. Note that by shrinking $\widetilde\delta$ as necessary we can ensure that $Y\cap \partial K=\emptyset$. Since $\gamma$ only has transversal reflections up to $x_1$ there are path-connected neighbourhoods $U_0\subseteq W_0$ of $\sigma(a)$ and $U_1\subseteq T_{\partial K} S_K$ of $(x_1,\omega_1)$ such that the map $\Lambda:U_0\to U_1$ given by the shift along the billiard flow is a diffeomorphism. Thus $\Lambda(U_0\cap P_\lambda^{(K,N)})$ is a codimension 1 submanifold of $U_1$ (\Cref{fig:u0u1}). We may shrink $\widetilde\delta$ so that $\widetilde X\subseteq \flow_\varepsilon(U_1)$. Set $U_i=\genflow_{\tau^*+\varepsilon}(U_1)$, it follows that $\widetilde Y\subseteq U_i$ (\Cref{fig:u1ui}). Note that $P_\lambda^* = \genflow_{\tau^*+\varepsilon}\circ\Lambda(U_0\cap P_\lambda^{(K,N)})$ is a codimension 1 submanifold of $U_i$. We claim that $\widetilde Y$ is transversal to $P_\lambda^*$.

	\begin{figure}
		\center
		\begin{tikzpicture}[scale=0.7]
			\coordinate (c1) at (-6,0);
			\draw[dashed] (c1) circle(3cm);
			\node at ($(c1) + (-45:3.5)$) {$U_0$};
			
			\coordinate (a1) at ($(c1) + (75:3.5)$);
			\coordinate (a2) at ($(a1) + (5,0)$);
			\coordinate (p1) at ($(c1) + (105:1.5)$);
			\coordinate (p2) at ($(c1) + (275:1.5)$);
			\coordinate (b1) at ($(c1) + (255:3.5)$);
			\coordinate (b2) at ($(b1) + (5,0)$);
			\draw[thick, name path = pr, <->] (a1) .. controls (p1) and (p2) .. (b1) node[below] {$P_\lambda^{(K,N)}$} node[sloped, midway, above right] {$$};
				
			\coordinate (s1) at ($(c1) + (0:2.5)$);
			\coordinate (s2) at ($(c1) + (180:0.5)$);
			\coordinate (sp1) at ($(c1) + (35:1.25)$);
			\coordinate (sp2) at ($(c1) + (-45:1.75)$);
			\draw[name path = sigma, opacity = 0] (s1) .. controls (sp1) and (sp2) .. (s2);
			\draw[name intersections = {of = sigma and pr, by={s0}}, opacity=0] (s1)--(s0);
			\fill (s0) circle (0.05cm) node[left] {$\sigma(a)$};
			
			\begin{scope}
				\clip (a1) .. controls (p1) and (p2) .. (b1) -- (b2) -- (a2) -- cycle;
				\draw [dashed, pattern=dots, pattern color=gray, opacity = 0.4] (c1) circle(3cm);
				\draw[<-] (s1) .. controls (sp1) and (sp2) .. (s2) node[midway, right] {$\sigma(s)$};
			\end{scope}
			\node at ($(c1) + (-60:2)$) {$W^*_0\cap U_0$};
			
			\coordinate (c2) at (3,0);
			\draw[dashed] (c2) circle(3cm);
			\node at ($(c2) + (-45:3.5)$) {$U_1$};
			
			\coordinate (aa1) at ($(c2) + (75:3)$);
			\coordinate (aa2) at ($(aa1) + (5,0)$);
			\coordinate (pp1) at ($(c2) + (105:1.5)$);
			\coordinate (pp2) at ($(c2) + (275:1.5)$);
			\coordinate (bb1) at ($(c2) + (255:3)$);
			\coordinate (bb2) at ($(bb1) + (5,0)$);
			\draw[thick, name path = ppr] (aa1) .. controls (pp1) and (pp2) .. (bb1) node[below] {$\Lambda(U_0\cap P_\lambda^{(K,N)})$};
				
			\coordinate (ss1) at ($(c2) + (0:2.5)$);
			\coordinate (ss2) at ($(c2) + (180:0.5)$);
			\coordinate (ssp1) at ($(c2) + (35:1.25)$);
			\coordinate (ssp2) at ($(c2) + (-45:1.75)$);
			\draw[name path = ssigma, opacity = 0] (ss1) .. controls (ssp1) and (ssp2) .. (ss2);
			\draw[name intersections = {of = ssigma and ppr, by={ss0}}, opacity=0] (ss1)--(ss0);
			\fill (ss0) circle (0.05cm) node[left] {$(x_1,u_1)$};
			
			\begin{scope}			
				\clip (aa1) .. controls (pp1) and (pp2) .. (bb1) -- (bb2) -- (aa2) -- cycle;
				\draw [dashed, pattern=dots, pattern color=gray, opacity = 0.4] (c2) circle(3cm);
				\draw[<-] (ss1) node[left, rotate=20] {$\Lambda(\sigma(s))$} .. controls (ssp1) and (ssp2) .. (ss2);
			\end{scope}
			\node at ($(c2) + (-60:2)$) {$\Lambda(W^*_0\cap U_0)$};
			
			\draw[bend left, ->, shorten <= 0.5cm, shorten >= 0.5cm] ($(c1) + (45:3)$) to ($(c2) + (135:3)$);
			\node at ($(c1) + (35:5.5)$) {$\Lambda$};
		\end{tikzpicture}
		\caption{The structure of $U_0$ and $U_1$}\label{fig:u0u1}
	\end{figure}
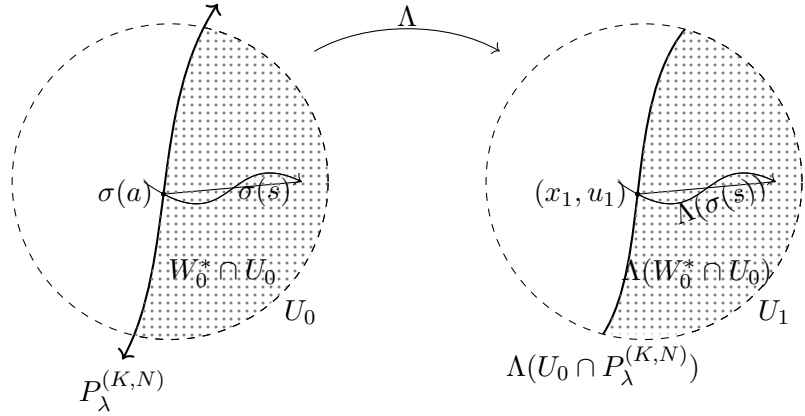
	
	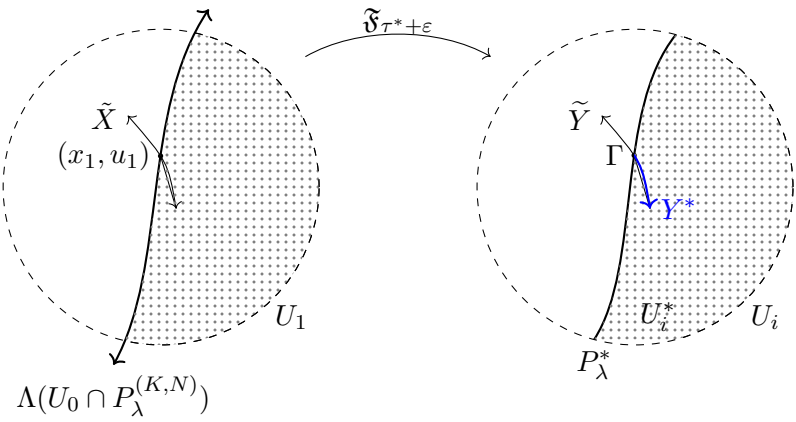
\begin{figure}
		\center
		\begin{tikzpicture}[scale=0.7]
			\coordinate (c1) at (-6,0);
			\draw[dashed] (c1) circle(3cm);
			\node at ($(c1) + (-45:3.5)$) {$U_1$};
			
			\coordinate (a1) at ($(c1) + (75:3.5)$);
			\coordinate (a2) at ($(a1) + (5,0)$);
			\coordinate (p1) at ($(c1) + (105:1.5)$);
			\coordinate (p2) at ($(c1) + (275:1.5)$);
			\coordinate (b1) at ($(c1) + (255:3.5)$);
			\coordinate (b2) at ($(b1) + (5,0)$);
			\draw[thick, name path = pr, <->] (a1) .. controls (p1) and (p2) .. (b1) node[below] {$\Lambda(U_0\cap P_\lambda^{(K,N)})$};
				
			\coordinate (s1) at ($(c1) + (-55:0.5)$);
			\coordinate (s2) at ($(c1) + (115:1.5)$);
			\coordinate (sp1) at ($(c1) + (75:0.5)$);
			\draw[<->, name path = sigma] (s1) .. controls (sp1) .. (s2) node[left] {$\tilde X$};
			\draw[name intersections = {of = sigma and pr, by={s0}}, opacity=0] (s1)--(s0);
			\fill (s0) circle (0.05cm) node[left] {$(x_1,u_1)$};
			
			\begin{scope}
				\clip (a1) .. controls (p1) and (p2) .. (b1) -- (b2) -- (a2) -- cycle;
				\draw [dashed, pattern=dots, pattern color=gray, opacity = 0.4] (c1) circle(3cm);
			\end{scope}
			\node at ($(c1) + (-80:2.5)$) {$$};
			
			\coordinate (c2) at (3,0);
			\draw[dashed] (c2) circle(3cm);
			\node at ($(c2) + (-45:3.5)$) {$U_i$};
			
			\coordinate (aa1) at ($(c2) + (75:3)$);
			\coordinate (aa2) at ($(aa1) + (5,0)$);
			\coordinate (pp1) at ($(c2) + (105:1.5)$);
			\coordinate (pp2) at ($(c2) + (275:1.5)$);
			\coordinate (bb1) at ($(c2) + (255:3)$);
			\coordinate (bb2) at ($(bb1) + (5,0)$);
			\draw[thick, name path = ppr] (aa1) .. controls (pp1) and (pp2) .. (bb1) node[below] {$P^*_\lambda$};

			\coordinate (ss1) at ($(c2) + (-55:0.5)$);
			\coordinate (ss2) at ($(c2) + (115:1.5)$);
			\coordinate (ssp1) at ($(c2) + (75:0.5)$);
			\draw[->, name path = ssigma] (ss1) .. controls (ssp1) .. (ss2) node[left] {$\widetilde Y$};
			\draw[name intersections = {of = ssigma and ppr, by={ss0}}, opacity=0] (ss1)--(ss0);
			\fill (ss0) circle (0.05cm) node[left] {$\Gamma$};
			
			\begin{scope}			
				\clip (aa1) .. controls (pp1) and (pp2) .. (bb1) -- (bb2) -- (aa2) -- cycle;
				\draw [dashed, pattern=dots, pattern color=gray, opacity = 0.4] (c2) circle(3cm);
				\draw[<->, thick, blue, name path = sigma] (ss1) node[right] {$Y^*$} .. controls (ssp1) .. (ss2);
			\end{scope}
			\node at ($(c2) + (-80:2.5)$) {$U_i^*$};
			
			\draw[bend left, ->, shorten <= 0.5cm, shorten >= 0.5cm] ($(c1) + (45:3)$) to ($(c2) + (135:3)$);
			\node at ($(c1) + (35:5.5)$) {$\genflow_{\tau^*+\varepsilon}$};
		\end{tikzpicture}
		\caption{Construction of $\widetilde X$, $\widetilde Y$ and $Y^*$}\label{fig:u1ui}
	\end{figure}

	Since $\gamma$ is tangent to $\partial K_{i^*}$ at $y_0$, the outward unit normal $N_0$ to $\partial K_{i^*}$ at $y_0$ is also normal to $\gamma$ at $y_0$. Let $N^*$ be the vector field along $\gamma$ given by parallel translation of $N_0$ along $\gamma$ from $y_0\in\partial K_{i^*}$ to $y_0^*\in Y$. Since $\gamma$ intersects $Y$ orthogonally (by construction) it follows that $N^*$ is tangential to $Y$ at $y^*_0$. For some small $\delta^*>0$ let $\alpha:(-\delta^*,\delta^*)\to Y$ be the smooth unit-speed geodesic in $Y$ such that $\alpha(0) = y_0^*$ and $\alpha'(0) = N^*$. Let $\widetilde\alpha(s) = (\alpha(s),\mu(\alpha(s))$ where $\mu(y)$ is the outward unit normal to $Y$ at $y$. Also let $\tau_0 = \tau - \tau^* - \varepsilon$, i.e $\tau_0$ is the time $\gamma$ takes to travel between $y_0^*$ and $y_0$. Setting $\beta(s) = \pr_1\circ\flow_{\tau_0}(\widetilde\alpha(s))$ we note that $\beta((-\delta^*,\delta^*))$ is diffeomorphic to $\alpha^*$ since $Y$ is strictly convex. Now provided $\delta^*$ is sufficiently small we must have $\beta((-\delta^*,0))\subseteq K_{i^*}$ and $\beta((0,\delta^*))\cap K_{i^*} = \emptyset$ since $\beta'(0) = N_0$. Consider the smooth geodesics
	\[
		r(s) = \{\pr_1\circ\flow_t(\widetilde\alpha(s)):0\leq t\leq\tau_0\}.
	\]
	For $s\in (-\delta^*,0)$, the smooth geodesics $r(s)$ must intersect $\partial K_{i^*}$. Moreover, $r(s)$ cannot be tangent to $\partial K_{i^*}$ provided it remains sufficiently close to $y_0$, owing to the strict convexity of $\partial K_{i^*}$. Furthermore when $s\in(0,\delta^*)$, since the Jacobi field of $r(s)$ at $r(0)$ is precisely $N^*$ (the vector field along $\gamma$ obtained from the outward unit normal to $\partial K_{i^*}$), and  since $r(0)$ is tangent to $\partial K_{i^*}$ it follows that $r(s)$ does not intersect $\partial K_{i^*}$ at all for any $s\in(0,\delta^*)$, provided $\delta^*$ is sufficiently small. It follows that the curve $\widetilde\alpha$ must be transversal to $P^*_\lambda$, and since $P^*_\lambda$ has codimension 1 in $U_i$, by shrinking $\widetilde\delta$ we can conclude that $\widetilde Y$ is transversal to $P^*_\lambda$ as claimed.
	
	Let $\gamma_1 = \gamma(t_1),\dots,\gamma_N =\gamma(t_N)$ be the points of reflection of $\gamma$ (possibly repeating the last point of reflection as needed until $N$ reflections are achieved). Denote by $j^*$ the integer such that $\gamma_{j^*} = y_0$. Then for each $j<j^*$, since $\gamma_j$ is a transversal point of reflection, there is a diffeomorphism $F_j:W_0\to\Omega_j$ (possibly shrinking $W_0$), where $\Omega_j\subseteq T_{\partial K} S_K$ is an open neighbourhood of $(\gamma_j,\dot\gamma_j(t_j))\in T_{\partial K} S_K$, such that $F_j$ is given by the shift along the billiard flow. Note that since $\gamma_{j^*}$ is a tangential point of reflection, there is no analogous diffeomorphism and open neighbourhood as there is for the transversal reflections of $\gamma$ prior to $\gamma_{j^*}$. We now consider two cases, both of which will produce diffeomorphisms $F_j$ and open neighbourhoods $\Omega_j$ for $j^*<j\leq N$, after which we shall proceed in the same manner. First suppose that there exists a sufficiently small $\delta$, such that for all $s\in(a-3\delta,a)$ the rays $\gamma_K^N(\sigma(s))$ do not intersect $\partial K_{i^*}$ near $y_0$. In this case, since each such ray $\gamma_K^N(\sigma(s))$ has no irregular points, there are diffeomorphism $\widetilde F_j:W_0^*\to\Omega_j$ for all $j^*\leq j \leq N$, with $\widetilde F_j$ and $\Omega_j$ defined analogously. We then define $F_j:\Omega_{j^*}\to\Omega_j$ as $F_j = \widetilde F_j \circ \widetilde F_{j^*}^{-1}$.
	In the second case, suppose that no such $\delta$ exists, hence for all $s\in(a-3\delta,a)$ the rays $\gamma_K^N(\sigma(s))$ intersect $\partial K_{i^*}$ near $y_0$. Since each such point of reflection is a regular point, there is a regular neighbourhood $V\subseteq\partial K_{i^*}$ such that $y_0\in\partial V$. Now there is a path-connected open neighbourhood $\Omega_{j^*}\subseteq T_{\partial K} S_K$ of $(y_0,u_0)\in\uT{\partial K_{i^*}}$ such that for each $j^*<j\leq N$ the map $F_j:\Omega_{j^*}\to\Omega_j$ is a diffeomorphism given by the shift along the billiard flow onto the neighbourhood $\Omega_j\subseteq T_{\partial K} S_K$ of $(\gamma_j,\dot\gamma_j(t_j))$. We shall now continue in the same manner for both cases, except when specified otherwise.
	
	Now there are $N-n$ regular points of reflection along $\gamma$ (possibly counting the last regular point multiple times). Denote these by $\gamma_{\zeta_1},\dots,\gamma_{\zeta_{N-n}}$. For each regular point there exists a regular neighbourhood $\Omega^*_{\zeta_j}\subseteq\partial K$. Then for each $\zeta_j$ we may shrink $\Omega_{\zeta_j}$ such that $\pr_1(\Omega_{\zeta_j})\subseteq\Omega_{\zeta_j}^*$. In doing so we may have to shrink $W_0$ and $\Omega_{j^*}$, as well as $\delta$ and $\widetilde\delta$ accordingly. Recall that $W_0\setminus P_\lambda^{(K,N)}$ has exactly two path components. Since $U_0\subseteq W_0$ it follows that $U_i\setminus P_\lambda^*$ has exactly two path components. Therefore $\widetilde Y\setminus P_\lambda^*$ also has exactly two path components, also owing to the fact that $\widetilde Y\cap P_\lambda^*$ has codimension 1 in $\widetilde Y$ since $P_\lambda^*$ is transversal to $\widetilde Y$. Recall that $\delta$ was chosen so that $\sigma(s)\in W_0^*$ for all $s\in [a-3\delta,a)$. If necessary, we shrink $\delta$ such that $\sigma(s)\in W_0^*\cap U_0$ for all $s\in [a-3\delta,a)$. Denote $U_i^* = \genflow_{\tau^*+\varepsilon}\circ\Lambda(W_0\cap U_0)$ (\Cref{fig:u1ui}). It follows that $U_i^*$ is one of the two path components of $U_i\setminus P_\lambda^*$. We must now deviate based on our two cases again. In the first case, by construction, every ray $\gamma_K^{N-j^*+1}(u)$ for $u\in U_i^*$ will not intersect $\partial K_{i^*}$ near $y_0$, and will pass through $N-j^*+1$ regular neighbourhoods, thus having no irregular points at all. In the second case, where the regular neighbourhood $V$ exists we proceed as follows. We may shrink $U_0$ and $U_1$ such that for any $u\in U_i^*$ there exists a minimum time $0<\ell_u<\varepsilon$ such that $\gamma_K^{N-j^*+1}(u)(\ell_u)\in V$. Note that such a ray may have some irregular points, however its first point of reflection will be a regular point in $V$ near $y_0$. Now in both cases, for any $u\in U_i^*$ the ray $\gamma_K^{N-j^*+1}(u)$ will not be tangent to $\partial K$ by our construction (since $u\not\in P_\lambda^*$). Now pick any $u^*\in Y^* = \widetilde Y\cap U_i^*$ (\Cref{fig:u1ui}) sufficiently close to $(y_0^*,\mu(y_0^*))\in\partial Y^*$. Set $\widetilde\sigma = \Lambda^{-1}\circ\genflow_{-\tau^*-\varepsilon}(u^*)\in W_0^*$. Recall that $W_0^*\cap(\cup_{\lambda'}P_{\lambda'}^{(K,N)})=\emptyset$, hence any $w\in W_0$ generates a ray $\gamma_K^N(w)$ which is not tangent to $\partial K$. Therefore, by our construction, the ray $\gamma_K^N(\widetilde\sigma)$ generates a ray which is not tangent to $\partial K$, has its first irregular point at $x_1$ and has at most $n-1$ irregular points. This follows since $\gamma_K^N(\widetilde\sigma)$ must pass through the $N-n$ regular neighbourhoods $\Omega_{\zeta_j}$ which we prescribed earlier. We may now smoothly extend $\sigma$ to a smooth path $\sigma^*:[0,a]\to\TTS$ by choosing an appropriate smooth path in $W_0^*$ from any $\sigma(s)\in W_0$ (picking $s\in[a-3\delta,a)$ appropriately) to $\widetilde\sigma$. It follows that $\sigma^*$ will be an $N$-admissible curve with the desired properties.
\end{proof}

\section{Eliminating Irregular Points}

For a given $N>1$ and $1\leq n \leq N$, define $\zn$ as the set of irregular points $x\in\partial K$ such that there exists an $N$-admissible path $\sigma:[0,a]\to\TTS$ with the following properties:
\begin{enumerate}
	\item $\gamma^+_K(\sigma(s))$ has at most $N$ reflections. \label{condition:non-trapped}
	\item $\gamma^N_K(\sigma(a))$ contains $x$. \label{condition:has-x}
	\item $\gamma^N_K(\sigma(s))$ has at most $n$ irregular points for all $s\in [0,a]$. \label{condition:n-irregular}
	\item $\gamma^N_K(\sigma(s))=\gamma^N_L(\sigma(s))$ for all $s\in [0,a]$. \label{condition:equal}
\end{enumerate}

\begin{proposition}[\cite{GNS2023Preprint1}]\label{prop:first_irregular_not_tangent}
	Suppose that $\sigma\in\TTS$ generates a ray $\gamma_K^N(\sigma)$ with irregular points. Then the first irregular point of $\gamma_K^N(\sigma)$ cannot be a tangent point of reflection.
\end{proposition}

Note that finding an $N$-admissible path satisfying condition \ref{condition:has-x} for any given $x\in\partial K$ is always possible since we can find a $p\in\TTS$ such that $\gamma^+_K(p)$ is non-trapped and contains $x$ (see \Cref{lemma:always_reachable}). Then we simply apply \Cref{proposition:N_admissible}. It is in fact conditions \ref{condition:non-trapped} and \ref{condition:equal} which are difficult to satisfy.

Since our goal is to ultimately show that there exist no irregular points on $\partial K$ (or $\partial L$), we will first show that the sets $\zn$ are empty. Although we will not directly show that every irregular point must be contained in some $\zn$, we will show that any irregular point must lead to some irregular point being contained in $\zn$, leading us to a contradiction. \Cref{prop:zn} is done via induction, and indeed showing that $\zN{n-1}=\emptyset$ implies $\zn=\emptyset$ will directly rely on \Cref{lemma:N-non_tangent,lemma:N-tangent_regular,lemma:N-tangent_irregular}. The idea is to take any $x\in\zn$ and reduce its corresponding $N$-admissible curve to one with fewer irregular points via one (or more) of the three lemmas, thereby reaching a contradiction. To conclude the induction, we show that $\zN{1}$ is empty by a direct application of the sets of traveling times of $K$ and $L$.

\begin{proposition}\label{prop:zn}
	For any given $N>1$ and $1\leq n < N$, we have $\zn = \emptyset$.
\end{proposition}
\begin{proof}
	First suppose that for all $1\leq i < n$ we have $\zN{i}=\emptyset$. Suppose that $\zn\neq\emptyset$ and take $x\in\zn$. Then there is some $N$-admissible path $\sigma:[0,a]\to\TTS$ such that $\gamma^+_K(\sigma(a))$ has at most $N$ reflections and contains $x$, and for all $s\in[0,a]$ the billiard rays $\gamma^N_K(\sigma(s))=\gamma^N_L(\sigma(s))$ have at most $n$ irregular points. Let
	\[
		b = \sup\{s^*\in[0,a]:\gamma^N_K(\sigma(s))\textrm{ has no irregular points for all }s\in[0,s^*)\}.
	\]
	Note that the ray $\gamma_b = \gamma^N_K(\sigma(b))$ must have some irregular points. Indeed, suppose that is not the case. Then every reflection point of $\gamma_b$ is contained in a regular neighbourhood. Hence for any $b'\in(b,a]$ sufficiently close to $b$, the ray $\gamma^N_K(\sigma(b'))$ cannot have any irregular points, since its reflection points will also be contained in the regular neighbourhoods of $\gamma_b$.  This contradicts the definition of $b$, hence $\gamma_b$ must have some irregular points. Recall that by definition $\gamma_b$ has at most $n$ irregular points. Suppose that $\gamma_b$ has $j\leq n-1$ irregular points. Pick any irregular point $x'\in\gamma_b$. Then $\sigma|_{[0,b]}$ is an $N$-admissible path satisfying the conditions for $x'\in\zN{n-1}$. But by assumption $\zN{n-1}=\emptyset$, thus we reach a contradiction. Hence $\gamma_b$ must have exactly $n$ irregular points.
	
	We note at this point that the first irregular point of $\gamma_b$ cannot be a tangent point of reflection with $\partial K$ by \Cref{prop:first_irregular_not_tangent}. We must now consider two possibilities, either $\gamma_b$ is tangent to $\partial K$ (exactly once) at a regular point, or it is not. Note that in the latter case $\gamma_b$ is either tangent to $\partial K$ at an irregular point, or is not tangent to $\partial K$ at all. We shall reduce the first case to the latter and then continue in the same manner for both cases. Hence, suppose that $\gamma_b$ is tangent to $\partial K$ at a regular point. For brevity, relabel $\sigma|_{[0,b]}$ as $\sigma$. Then $\sigma$ satisfies all the conditions of \Cref{lemma:N-tangent_regular}. Therefore there is an $N$-admissible path $\sigma^*:[0,b]\to\TTS$ and $\delta^*>0$ such that $\sigma^*(s)=\sigma(s)$ for all $s\in[0,b-3\delta^*)$ and $\gamma_K^N(\sigma^*(s))$ contains at most $n$ irregular points for all $s\in[b-3\delta^*,b]$ and $\gamma_K^N(\sigma^*(b))$ is a non-tangent ray containing the first irregular point of $\gamma_b$. Furthermore, note that $\gamma^+_K(\sigma^*(s))$ has at most $N$ reflections for all $s\in[0,b]$. We can now move on to the latter case.
	
	In the latter case, where $\gamma_b$ has no tangential points of reflection on $\partial K$ or has a single tangential reflection with $\partial K$ at an irregular point we simply set $\sigma^* = \sigma|_{[0,b]}$. We may now proceed from both cases in the same manner using $\sigma^*$ as our $N$-admissible path satisfying the conditions of $\zn$. By \Cref{lemma:N-non_tangent,lemma:N-tangent_irregular} we can find $\delta>0$ sufficiently small and construct a new $N$-admissible path $\widetilde\sigma:[0,b]\to\TTS$ such that
	\begin{itemize}
		\item $\widetilde\sigma(s) = \sigma^*(s)$ for all $s\in[0,b-3\delta)$.
		\item $\gamma^+(\widetilde(\sigma(s))$ has at most $N$ reflections for all $s\in[0,b]$.
		\item $\gamma_K^N(\widetilde\sigma(s))$ has at most $n-1$ irregular points for all $[b-3\delta,b]$.
		\item $\gamma_K^N(\widetilde\sigma(b))$ is not tangent to $\partial K$.
		\item The first irregular point of $\gamma_K^N(\widetilde\sigma(b))$ is the same as the first irregular point of $\gamma_b$.
	\end{itemize}
	Recall that $\gamma_K^N(\sigma^*(s))=\gamma_L^N(\sigma^*(s))$ for all $s\in[0,b]$. However, the same may not be true for $\widetilde\sigma$. Note that by construction this can only occur when $s\in[b-3\delta,b]$. Thus we set
	\[
		c = \sup\{s^*\in[b-3\delta,b]:\gamma_K^N(\widetilde\sigma(s))=\gamma_L^N(\widetilde\sigma(s))\textrm{ for all }s\in[b-3\delta,s^*)\}.
	\]
	Suppose that $c<b$, and denote $\gamma_c = \gamma_K^N(\widetilde\sigma(c))$. Now consider any $s\in[b-3\delta,c)$. Since the ray $\gamma^N_K(\widetilde\sigma(s)) = \gamma_L^N(\widetilde\sigma(s))$ has at most $n-1$ irregular points, they must all belong to $\zN{n-1} = \emptyset$. Thus every reflection point of $\gamma_K^N(\widetilde(\sigma(s))$ is a regular point. Let $W_j(s)\subseteq\partial K$ denote the maximal regular neighbourhood of the $j$-th reflection point of $\gamma_K^N(\widetilde\sigma(s))$. Given $s$ is sufficiently close to $c$ it follows that every point of reflection of $\gamma_c$ (apart from possibly the single tangential reflection which $\gamma_c$ is permitted to have) must be contained in $\overline{W_j(s)}$ for some $j\geq N$. By continuity $\overline{W_j(s)}\cap\partial L = \overline{W_j(s)}$ and the outward unit normal fields of $\partial L$ and $\partial K$ must agree on $\overline{W_j(s)}$ for all $j$. Hence $\gamma_K^N(\widetilde\sigma(c))=\gamma_L^N(\widetilde(\sigma(c))$, so $\widetilde\sigma|_{[0,c]}$ satisfies the conditions of $\zN{n-1}$. Therefore $\gamma_c$ can only have regular points of reflections, owing to the fact that $\zN{n-1}=\emptyset$. Let $W_j\subseteq\partial K$ be the regular neighbourhood of the $j$-th reflection point of $\gamma_c$. Then for $s^*>c$ sufficiently close to $c$, every point of reflection of $\gamma_K^N(\widetilde\sigma(s))$ where $s\in[c,s^*]$ (apart from possibly a tangential reflection) must be contained in $W_j$ for some $j\geq N$. Hence by the argument above, $\gamma_K^N(\widetilde\sigma(s)) = \gamma_L^N(\widetilde\sigma(s))$ for all $s\in[b-3\delta,s^*]$. This is a contradiction with the definition of $c$. Hence our assumption that $c<b$ was false. Suppose that $c=b$. Then as in our argument above, $\gamma_K^N(\widetilde\sigma(b)) = \gamma_L^N(\widetilde\sigma(b))$. Therefore $\widetilde\sigma$ satisfies the conditions for $\zN{n-1}$. However, by construction $\gamma_K^N(\widetilde\sigma(b))$ contains $x^*\in\zn$, the first irregular point of $\gamma_b$. This is a contradiction, since this would imply that $x^*\in\zN{n-1}$, but $\zN{n-1}=\emptyset$. Hence it follows that $\zn=\emptyset$ as desired, assuming that $\zN{i}=\emptyset$ for all $1\leq 1 \leq n-1$. To complete our induction we must now show that $\zN{1} = \emptyset$.
	
	Suppose that $x\in\zN{1}$, then there exists an $N$-admissible path $\sigma:[0,a]\to\TTS$ such that
	\begin{enumerate}
	\item $\gamma^+_K(\sigma(s))$ has at most $N$ reflections. \label{condition:non-trapped}
	\item $\gamma^N_K(\sigma(a))$ contains $x$. \label{condition:has-x}
	\item $\gamma^N_K(\sigma(s))$ has at most $1$ irregular point for all $s\in [0,a]$. \label{condition:n-irregular}
	\item $\gamma^N_K(\sigma(s))=\gamma^N_L(\sigma(s))$ for all $s\in [0,a]$. \label{condition:equal}
	\end{enumerate}
	Note that by \Cref{prop:first_irregular_not_tangent} for any $s\in[0,a]$ the irregular point of $\gamma_K^N(\sigma(s))$ (if it exists) cannot be a tangent point of reflection with $\partial K$. Now set
	\[
		b = \sup\{s^*\in[0,a]:\gamma^N_K(\sigma(s))\textrm{ has no irregular points for all }s\in[0,s^*)\}.
	\]
	Then as before, $\gamma_b=\gamma^N_K(\sigma(b))$ must have exactly one irregular point $x^*\in\zN{1}$. Let $x_1\dots x_p$, where $p\leq N$, be the points of reflection of $\gamma_b$ (in order of appearance along $\gamma_b$). Denote by $i^*$ the integer such that $x_{i^*}=x^*$ is the irregular point of $\gamma_b$. Then for each $i\neq i^*$ there is a maximal regular neighbourhood $W_i\subseteq\partial K$ of $x_i$. For $x_{i^*}$, since it is a transversal point of reflection, given $s\in[0,b)$ sufficiently close to $b$, we must have $\gamma_K^N(\sigma(s))$ reflect transversally on $\partial K$ near $x_{i^*}$. Thus for such $s$ there is a maximal regular neighbourhood $W_{i^*}$, since $\gamma_K^N(\sigma(s))$ has no irregular points. It follows that $x_{i^*}\in\overline{W_{i^*}}$ provided $s$ is chosen sufficiently close to $b$. Therefore $x_{i^*}\in\partial L$ since $\overline{W_{i^*}}\subseteq\partial L$ by continuity. 
	Since $\gamma_K^+(\sigma(b))$ has exactly $p$ reflections with $\partial K$ there is a neighbourhood $U\in\TTS$ of $\sigma(b)$ such that for every $u\in U$ the ray $\gamma_K^+(u)$ has either $p$ or $p-1$ reflections. Note that if $\gamma_b$ is not tangent to $\partial K$ then we can choose $U$ such that $\gamma_K^+(u)$ will have exactly $p$ reflections for all $u\in U$. We can shrink $U$ such that for every $u\in U$ the ray $\gamma_K^+(u)$ will only have reflections in $W_i$ for $i\neq i^*$ except near $x_{i^*}$. Given $u=(x_0,\omega_0)\in U$, let $(x_1^{(K)},\omega_1^{'(K)}),\dots (x_{p'}^{(K)},\omega_{p'}^{(K)})$, where $p' = p$ or $p-1$, be the points of reflection and reflected directions of $\gamma^+_K(u')$. Note that provided we choose $U$ sufficiently small $\gamma_L^+(u)$ must have the same number of reflections. Hence let $(x_1^{(L)},\omega_1^{(L)}),\dots (x_{p'}^{(L)},\omega_{p'}^{(L)})$ be the points of reflection and reflected directions of $\gamma^+_K(u')$. Let $\tau>0$ be the travelling time of $\gamma^+_K(u')$, and set $(x_{p'+1},-\omega_{p'+1})=\genflow^{(K)}_\tau(u)$. Now since the sets of travelling times of $K$ and $L$ are equal it follows that $\genflow^{(K)}_\tau(u)=\genflow^{(L)}_\tau(u)$. Denote the reflection of $\gamma^+_K(u')$ near $x_{i^*}$ by $x_{i'}$ for some $1\leq i'\leq p'$.
	Now by construction, for each $i\neq i'$, we have $x'_i\in W_j$ for some $j\neq i^*$. Hence it follows that for $1\leq i<i'$ we must have $(x_{i}^{(K)},\omega_{i}^{(K)})=(x_{i}^{(L)},\omega_{i}^{(L)})$ and similarly for $i>i'$ we must also have $(x_{i}^{(K)},\omega_{i}^{(K)})=(x_{i}^{(L)},\omega_{i}^{(L)})$. It follows that $(x_{i'-1}^{(K)},\omega_{i'-1}^{(K)})=(x_{i'-1}^{(L)},\omega_{i'-1}^{(L)})$ and $(x_{i'+1}^{(K)},\omega_{i'+1}^{(K)})=(x_{i'+1}^{(L)},\omega_{i'+1}^{(L)})$. Therefore it must be the case that $(x_{i'}^{(K)},\omega_{i'}^{(K)})=(x_{i'}^{(L)},\omega_{i'}^{(L)})$. Hence there is a neighbourhood of $x_{i^*}$ where $\partial K = \partial L$, contradicting the fact that $x_{i^*}$ is an irregular point. Thus we can conclude that $\zN{1} = \emptyset$.
	
	Therefore by induction it follows that $\zn = \emptyset$ for all $1\leq n \leq N$, as claimed.
\end{proof}

The trapping set $\Trap{\partial S}^{(K)}$ presents a challenge to our approach, since \Cref{prop:zn} relies directly on the information given by the travelling times to show that one cannot have a single irregular point along a ray. This of course is not possible when the ray in question is trapped, since there is no travelling time data for such rays. When the dimension $m$ of $S$ is greater than 2, it turns out that the complement of the trapping set is path-connected. 

\begin{proposition}\label{prop:ttts-path-connected}
	$\TTTS$ is path-connected when $m\geq 3$.
\end{proposition}
\begin{proof}
	Denote by $B_r\subseteq S$ the ball of radius $r>0$ with respect to the metric on $S$. Let
	\[
		\widetilde\delta = \inf\{r\in\reals: K\subseteq B_r\}.
	\]
	Given $x\in\partial S$, let $(x,\omega_0)\in\TS$ be any vector such that the ray $\gamma^+_K(x,\omega_0)$ does not intersect $B_{\widetilde\delta}$. To prove the result it is sufficient to show that for any $(x,\omega_1)\in\uT{S}_x\cap \TTTS$ there exists a continuous path $\sigma:[0,a]\to\TTTS$ such that $\sigma(0)=(x,\omega_0)$ and $\sigma(a) = (x,\omega_1)$. Denote the inward unit normal to $\partial S$ at $x$ by $N_S(x)$. Set
	\[
		\widehat\delta = \inf\{\norm{\omega-N_S(x)}:(x,\omega)\in\TS\cap\uT{S}_x\textrm{ and }\gamma^+_K(x,\omega)\cap B_{\widetilde\delta} \neq\emptyset\}.
	\]
	\[
			\Sigma = \{\exp_{x}(\varepsilon\omega):(x,\omega)\in\TS\cap\uT{S}_x\textrm{ and }\norm{\omega-N_S(x)}<\delta\}.
	\]
	\[
			\widetilde\Sigma = \{(y,n(y)):y\in\Sigma\textrm{ where }n(y)\textrm{ is the outward unit normal to $\Sigma$ at }y\}.
	\]
	We pick $\varepsilon>0$ such that the minimum principal curvature of $\Sigma$ is greater than $\globalcurv$ (\cite{MR618545}). We may pick $\delta>\widehat\delta$ such that $\Sigma\cap\partial S=\emptyset$ (possibly shrinking $\varepsilon$). Note that $\Sigma$ (and $\widetilde\Sigma$) are diffeomorphic to an open neighbourhood $U$ of $N_S(x)$ in $\uT{S}_x\cap \TS$, we shall denote this diffeomorphism by $\Phi:\Sigma\to U$. Let $\Sigma(n,i)\subseteq\Sigma$ be the set of points $y\in\Sigma$ such that the ray $\gamma_K^+(y,n(y))$ is trapped, tangent to $\partial K$ exactly once, after $n$ reflections and the tangent point of reflection is on $\partial K_i$. In \cite{GNS2023Preprint1} it is shown that the topological dimension of $\Sigma(n,i)$ is 0. Now denote by $\Sigma_\circ\subseteq\Sigma$ the set of points $y\in\Sigma$ such that the ray $\gamma_K^+(y,n(y))$ is trapped but never tangent to $\partial K$. Given $y\in\Sigma_\circ$ let $(x,\omega_y) = \Phi(y)$. Let $U_y\subseteq U$ be any path-connected neighbourhood of $(x,\omega_y)$ such that $\gamma_K^+(x,\omega)$ is not tangent to $\partial K$ for all $(x,\omega)\in U_y$. Set $V_y = \Phi^{-1}(U_y)$, then we claim that $\dim(V_y\cap\Sigma_\circ)=0$. Recall that $d$ is the number of components of $K$, and set $D_0 = \{1,2,\dots,d\}$ and $D = \prod_{i=1}^\infty D_0$. We endow $D_0$ with the discrete topology and $D$ with the product topology. It is well known that $D$ has topological dimension 0 (\cite{MR0482697}). Denote $V_y^\circ = V_y\cap\Sigma_\circ$ and let $\Psi:V_y^\circ\to D$ be the map defined as follows. $\Psi(y)=(i_1,i_2,\dots)$ if $\gamma^+_K(y,n(y))$ has its $j$-th reflection on the component $\partial K_{i_j}$. It is shown in \cite{GNS2023Preprint1} that $\Psi$ is in fact a homeomorphism onto its image, hence $\dim(V_y^\circ)=0$ as claimed. It follows that $\dim(\Sigma_\circ)=0$ since for every $y\in\Sigma_\circ$ there is an open set $V_y\subseteq\Sigma$ such that $y$ is contained in the corresponding subset $V_y^\circ$ with topological dimension $0$. Therefore by the sum theorem (\cite{MR0482697}) it follows that $\dim(\cup_n\cup_i\Sigma(n,i)\cup\Sigma_\circ)=0$. Now we simply note that $U\cap\Trap{\partial S}^{(K)}\subseteq \cup_i\Xi_i\cup\Phi(\cup_n\cup_i\Sigma(n,i)\cup \Sigma_\circ)$ so $U\setminus\Trap{\partial S}^{(K)}$ is path-connected, as desired, and this completes our proof.
\end{proof}
Note that by using the argument above at a point $x\in\partial K$ we can immediately find the following conclusion, which holds for all $m\geq 2$.
\begin{corollary}\label{lemma:always_reachable}
	For every $x\in\partial K$ there is some $\sigma\in\TTTS$ such that $x\in\gamma^+_K(\sigma)$.
\end{corollary}

\section{Proof of Theorem 1}

\begin{theorem}\label{thm:one-sided}
	Suppose that $K$ and $L$ have the same set of travelling times, and for every point $x\in\partial K$ there is an $N$-admissible path $\sigma:[0,a]\to\TTS$ such that:
	\begin{enumerate}
		\item $\gamma_K^+(\sigma(a))$ contains $x$.
		\item If $\gamma_K^+(\sigma(s))$ has more than $N$ reflections then it has no irregular points.
	\end{enumerate}
	Then $K\subseteq L$.
\end{theorem}
\begin{proof}
	Suppose that there exists $x\in\partial K\setminus\partial L$. Then by our hypothesis there is an $N$-admissible path $\sigma:[0,a]\to\TTS$ such that $\gamma_K^+(\sigma(a))$ contains $x$ and whenever $\sigma(s)\in \Trap{\partial S}^{(K)}$ then $\gamma_K^+(\sigma(s))$ has no irregular points. Let
	\[
		b = \sup\{s^*\in[0,a]:\gamma_K^+(\sigma(s))\textrm{ has no irregular points for all } s\in[0,s^*)\}.
	\]
	Note that as in the proof of \Cref{prop:zn}, $\gamma_b=\gamma_K^+(\sigma(b))$ must have at least one irregular point. It therefore follows, owing to our hypothesis, that $\sigma(b)\not\in \Trap{\partial S}^{(K)}$. In fact, it follows that $\gamma_K^+(\sigma(b)$ has at most $N$ reflections. If there is some $s\in(0,b)$ such that $\sigma(s)\in\Trap{\partial S}^{(K)}$ we set
	\[
		b' = \inf\{s^*\in[0,b):\gamma^+_K(\sigma(s))\textrm{ has at most } N \textrm{ reflections for all } s\in(s^*,b]\}.
	\]
	It follows that for $s\in(b',b]$ we have $\sigma(s)\in\TTTS$. Note that by the definitions of $b$ and $b'$ the rays $\gamma_K^+(\sigma(s))$ have at most $N$ reflections and contain no irregular points for all $s\in(b',b)$. Thus pick any $c\in(b',b)$. We can now re-parameterise and relabel $\sigma|_{[c,b]}$ as $\sigma:[0,a']\to\TTTS$ for some $a'\in\reals$. We now set
	\[
		c' = \sup\{s^*\in[0,a']:\gamma_K^N(\widetilde\sigma(s))=\gamma_L^N(\widetilde\sigma(s))\textrm{ for all }s\in[0,s^*)\}.
	\]
	Then repeating the argument in the proof of \Cref{prop:zn}, it follows that $c' = a'$, since the rays $\gamma_K^N(\sigma(s))$ contain only regular points for all $s<a$. Now let $x'$ be the first irregular point of $\gamma_{a'}=\gamma_K^N(\sigma(a'))$, and let $n$ be the number of irregular points along $\gamma_{a'}$. It follows that $x'\in\zn$, but by \Cref{prop:zn} we have $\zn=\emptyset$. Hence we reach a contradiction, and thus $K\subseteq L$.
\end{proof}

In the case where $m\geq 3$, given $x\in\partial K$ we can always find an $N$-admissible path $\sigma:[0,a]\to\TTS$ such that $\gamma_K^+(\sigma(a))$ contains $x$, owing to \Cref{proposition:N_admissible}. Furthermore following from \Cref{prop:ttts-path-connected} we can ensure that $\sigma(s)\in\TTTS$ for all $s\in[0,a]$. Since $N$ is chosen arbitrarily, we may simply set it to be the maximal number of reflections for $\gamma_K^+(\sigma(s))$. It follows that the requirements for \Cref{thm:one-sided} are trivially satisfied. By symmetry the same holds for any $x\in\partial L$. Hence we have proved the following.

\mainthm*

\end{document}